\documentclass[12pt,reqno]{amsart}

\usepackage{a4wide}

\usepackage{amsmath,amssymb,amsfonts,amsthm,amscd}
\usepackage{mathrsfs}
\usepackage{stmaryrd}

\usepackage{graphicx}
\usepackage{url}
\usepackage{indentfirst}
\usepackage{enumerate}
\usepackage{appendix}
\usepackage{comment}
\usepackage[numbers,sort&compress]{natbib}
\usepackage{color}

\numberwithin{equation}{section}
\allowdisplaybreaks[1]

\usepackage[colorlinks, linkcolor=blue, citecolor=blue]{hyperref}
\newcommand{\e}{\varepsilon}
\newcommand{\var}{\varepsilon}

\newcommand{\zhuxiang}{U_{\rho,\varepsilon}\Big(\frac{r}{\e}-\rho\Big)}
\newcommand{\zhuxianga}{\Big(U_{\rho,\varepsilon}\Big(\frac{r}{\e}-\rho\Big)\Big)}
\newcommand{\yuxianga}{\omega_{\rho,\varepsilon}\Big(\frac{r}{\e}\Big)}
\newcommand{\wucha}{\sqrt{\e^n\rho^{n-1}+\e^n}}
\newcommand{\er}{\Big(1+\e^2V(\e\rho)\Big)}
\newcommand{\ber}{\beta_{\e,\rho}}

\newcommand{\R}{\mathbb{R}}

\newcommand{\be}{\begin{equation}}
\newcommand{\ee}{\end{equation}}
\newcommand{\bs}{\begin{split}}
\newcommand{\es}{\end{split}}

\newtheorem{Thm}{Theorem}[section]
\newtheorem{Lem}[Thm]{Lemma}
\newtheorem{lemma}[Thm]{Lemma}

\newtheorem{Prop}[Thm]{Proposition}
\newtheorem{proposition}[Thm]{Proposition}

\newtheorem{Rem}[Thm]{Remark}
\newtheorem{remark}[Thm]{Remark}

\newcounter{specialprop}
\newtheorem{specialproposition}[specialprop]{Proposition}

\begin{document}

\title{Asymptotic Sphere Concentration at Infinity for NLS with  $L^2$ Constraint}
\author{Qing Guo, Chongyang Tian}
\address[Qing Guo]{College of Science, Minzu University of China, Beijing 100081, China}
\email{guoqing0117@163.com}
\address[Chongyang Tian]{School of Mathematics and Statistics, Central China Normal University, Wuhan 430079, China}
\email{cytian@mails.ccnu.edu.cn}

\maketitle

\begin{abstract}
	We consider the nonlinear Schr\"odinger equation
	\[
	-\Delta u + V(x)\,u = a\,u^p + \mu u \quad \text{in }\mathbb{R}^n,
	\qquad \int_{\mathbb{R}^n} u^2 = 1,
	\]
	modeling attractive Bose--Einstein condensates. For all dimensions $n\ge 2$ and all exponents $p>1$, we prove the existence of normalized solutions whose $L^2$-mass concentrates on spheres with radii diverging to infinity. In particular, the concentration set escapes to infinity rather than remaining on a fixed compact hypersurface, which makes our regime qualitatively different both from classical point-concentration phenomena and from concentrating profiles in unconstrained problems. Our approach combines a tailored finite-dimensional reduction with a blow-up analysis based on Pohozaev identities and, in this way, extends the two-dimensional mass-critical result for $(n,p)=(2,3)$ obtained in Guo--Tian--Zhou (Calc.\ Var.\ Partial Differential Equations, 2022). The proof in that paper relies in an essential way on the two-dimensional structure and does not directly apply in higher dimensions, whereas here we develop a different approximation scheme and functional setting adapted to the high-dimensional sphere-at-infinity concentration regime.
	\end{abstract}
\small{
		\smallskip

\section{introduction}
\subsection{Backgrounds}
We consider the following nonlinear eigenvalue problem
\begin{equation}\label{eq1}
-\Delta u +V(x) u = a u^p +\mu u,~\text{in}\;\mathbb R^n,
\end{equation}
under the $L^2$-constraint
\begin{equation}\label{eq1'}
\int_{\mathbb R^n} u^2=1,
\end{equation}
where $p>1$. The equations \eqref{eq1}-\eqref{eq1'} stem from the well-known
Bose-Einstein condensates (BEC). Traced back to 1920's,  Einstein predicted that, below a critical temperature, part of the bosons would occupy the same quantum state to form a condensate.
Henceforth, the research on BEC has attracted a large number of mathematicians and physicists. Over the last two decades, significant experiments on BEC in dilute gases of alkali atoms \cite{Anderson,Bloch,Davis} revealed various interesting quantum phenomena.
One can refer to the Nobel lectures \cite{Cornell,Ketterle} for more understanding of BEC.
\smallskip

New experimental advancements make the mathematicians devote themselves to  the   Gross-Pitaevskii (GP) equations proposed by Gross  and Pitaevskii \cite{Gross,Pitaevskii} in the 1960s:
\begin{equation}\label{1-22-5}
i \partial_t \psi(x, t)= -\Delta \psi(x, t) +V(x) \psi(x, t) -a |\psi(x, t)|^2\psi(x, t),\,\,x\in \mathbb R^2,
\end{equation}
with the $L^2$-constraint
\begin{equation*}
\int_{\mathbb R^2} |\psi(x, t)|^2\,dx=1,
\end{equation*}
where $i$ is the imaginary unit, $V(x)\ge 0$ is a real-valued potential and $a\in \mathbb R$  is an arbitrary
parameter. %For further results on BEC and Gross-Pitaevskii equations, one can refer to \cite{Lieb2} and  the references therein.
To search for a standing wave for the equation \eqref{1-22-5} of the form $
\psi(x, t)= u(x) e^{-i \mu t}$ with
$u(x)$ independent of time $t$ and $\mu$ representing the chemical potential of the condensate,  we obtain a nonlinear eigenvalue problem
\begin{equation}\label{eq0}
-\Delta u +V(x) u = a u^3 +\mu u,~\text{in}\;\mathbb R^2,
\end{equation}
with the $L^2$-constraint
\begin{equation}\label{eq0'}
\int_{\mathbb R^2} u^2=1,
\end{equation}
which is a special mass-critical case of \eqref{eq1}-\eqref{eq1'} with $p=3$ and $n=2$.
\smallskip

As for the problem \eqref{eq0}-\eqref{eq0'}, the ground state solutions and their properties have been the most well studied by far, see \cite{GS,Guo,Guo1,Bao}.  Luo, Peng, Wei and Yan in \cite{Luo} use the finite-dimensional reduction method to extended the asymptotic behaviors results on the ground state solutions to the exited states of problem \eqref{eq1}--\eqref{eq1'} with $p=3$ and $n=2,3$.

With the ring–shaped radial potential $V(x)=(|x|-1)^2$, Guo, Zeng and Zhou \cite{Guo1}
established the existence and asymptotic behavior of ground states for
\eqref{eq0}–\eqref{eq0'}. Although $V$ is radial, the minimizers found in \cite{Guo1}
are nonradial, since their concentration points lie near the unit circle in
$\mathbb{R}^2$. The existence of \emph{radial} solutions—necessarily excited
states—had remained unclear. Recently, Guo, Tian and Zhou \cite{gtzcv} obtained the first rigorous example
of ring-shaped normalized states in the radial, mass-critical setting. They
considered the normalized problem \eqref{eq0}--\eqref{eq0'} and proved that if
a radial solution $u_a$ concentrates at some radius $r_a$ as $a \to a_0$, then
necessarily $a_0 = +\infty$ and $r_a \to +\infty$; moreover, such ring-like
profiles can be constructed under suitable nondegeneracy assumptions. Their
analysis, however, is restricted to the two-dimensional mass-critical exponent
$p-1 = \tfrac{4}{n}$ with $n=2$, and it is far from obvious a priori whether
the same phenomenon persists for general $p>1$ and higher dimensions.
\smallskip

The present paper shows that this phenomenon is in fact robust in the radial
class. For arbitrary dimensions $n \ge 2$ and any $p>1$, we construct radial
normalized solutions whose mass concentrates on spheres whose radii diverge to
infinity as the mass parameter $a \to +\infty$. Furthermore, the parameter $a$
and the associated Lagrange multiplier $\mu$ admit a unified asymptotic
description that does not depend on whether the problem is mass-subcritical,
mass-critical or mass-supercritical. This stands in sharp contrast with the
nonradial case, where the limiting behavior changes substantially across these
three regimes; see, for instance, \cite{Guo2,Guo3}.

Another useful way to interpret our results is to compare the normalized problem with
its unconstrained counterpart. In the unconstrained setting, singularly
perturbed equations of the form
\[
-\varepsilon^2 \Delta u + V(|x|)u = u^p, \qquad u \in H^1(\mathbb{R}^n),\ p>1,
\]
admit positive radial solutions concentrating on fixed spheres $r = r_0$; see
\cite{cmp}. One can find more   high dimensional concentration phenomenon
for classical Schrödinger equations in \cite{BadialeDAprile2002,BP,BP1,BP2,BW06,DY10,Musso,Musso,Wei} and references therein. We would like to point out that the presence of the $L^{2}$-constraint
qualitatively changes this picture, and several additional observations are
required beyond those used for the classical Schrödinger equation.	
Without
the constraint, concentrating profiles typically localize at finitely many
points or on spheres determined by the geometry of $V$ and the perturbation.
When the total $L^2$–mass is fixed, such localization in a bounded region is
ruled out: the mass cannot be fully accommodated near a fixed radius, and
concentrating profiles are forced to drift to infinity and localize on spheres
whose radii $r_a$ diverge. Thus the constraint transforms fixed spherical
shells into expanding spherical layers escaping to infinity.
\smallskip

On the other hand, normalized solutions with pointwise concentration have been constructed in
several special regimes, such as the mass-critical and mass-supercritical
cases in low dimensions \cite{Luo}, the Schr\"odinger--Newton system in a
mass-subcritical range \cite{Guo3}, and fractional Schr\"odinger equations
covering all admissible relations between $n$ and $p$ \cite{Guo2}. By
comparison, normalized solutions exhibiting genuinely higher-dimensional
concentration (such as rings or spheres) are much less understood: apart from
the two-dimensional case $(p=3,n=2)$ treated in \cite{gtzcv}, we are not aware
of further results of this type. Our work provides a systematic extension to
$p>1$ and $n \ge 2$ under radial symmetry.
\smallskip

From a methodological perspective, the fact that we can treat all exponents
$p>1$ hinges both on the radial structure of the potential and on the specific
geometry of solutions concentrating on a sphere. Passing to the radial form of
\eqref{eq1}--\eqref{eq1'} allows us to exploit symmetry and to perform a local
finite-dimensional reduction along the manifold of spherical concentration
profiles. This reduction is carried out in a thin neighbourhood of the
concentrating sphere and therefore does not rely on any global compactness of
the Sobolev embedding, so no subcriticality assumption on $p$ is needed. As a
consequence, under a mild structural (nondegeneracy) condition on the
effective radial profile $M$, in the spirit of \cite{cmp}, we obtain existence
and a precise asymptotic description of the concentrating radius for every
$p>1$, including the energy-critical and supercritical ranges. For comparison,
we recall that point-concentrating spikes are known not to exist at the
energy-critical exponent $p = \tfrac{n+2}{n-2}$, see \cite{bucunzai}, which
highlights how spherical concentration at infinity circumvents the
obstructions inherent to point-concentration constructions.

\subsection{Main results}

As a preceding result, we know from \cite{Luo} that if $u_a$ is a solution of
problem \eqref{eq1}-\eqref{eq1'} which is concentrated at some points as $a \to a_0$, then it holds that
\[
\mu = \mu_a \to -\infty \quad \text{as } a \to a_0\in[0,+\infty].\]

%Denote  \be\label{Ma}M_a(r) := (-\mu_a)^{-(n-2)}r^{n-1}\left(1-\frac1{\mu_a}V(r)\right)^{\frac{p+3}{2(p-1)}}.\ee
\smallskip

Define $\varepsilon =\frac{1}{\sqrt{-\mu_a}}$ and change $u(x)$ to $\left(\frac{-\mu_a}{a}\right)^{\frac1{p-1}} u(x)$.
Then, \eqref{eq1}--\eqref{eq1'} can be written as the following  singularly perturbed elliptic problem on \(\mathbb{R}^{n}\) with radial potentials 
\begin{equation} \label{fangcheng1}
\displaystyle- \e^2\Delta u_\e+(1+\e^2 V(|x|) )u_\e= u_\e^{p}, \quad u_\e\in H^{1}(\mathbb{R}^{n}),\ee
with the constraint
\be\label{fangcheng11}
\displaystyle\int_{\mathbb{R}^n}u_\e^2(x)dx=a^{\frac{2}{p-1}}\e^{\frac{4}{p-1}},
\ee
where \(|x|\) denotes the Euclidean norm of \(x \in \mathbb{R}^{n}\), \(V: \mathbb{R}^{+} \to \mathbb{R}\) and \(p>1\). Suppose there exists a positive constant $\lambda_0$, for $\e$ small enough, $\inf\{1+\e^2V(\e r) : r \in \mathbb{R}^{+}\}\geq\lambda_{0}^{2}> 0$. 
\smallskip

For any $u\in H^1_r\left(\R^n\right)$, we equip it with the following norm
\begin{equation*}
\|u\|_\varepsilon:=\left(\displaystyle\int_{0}^{+\infty}r^{n-1}
\left(\varepsilon^2|u'(r) |^2+u^2(r)\right)dr\right)^{\frac{1}{2}}.
\end{equation*}An important property is the following: if $u \in H_{r}^{1}\left(\R^n\right)$, then  
\begin{equation}\label{jingxiangjiedeguji}
|u(r)| \leq C\, r^{\tfrac{1-n}{2}} \, \|u\|_{H^{1}(\mathbb{R}^{n})}, \quad r \geq 1,
\end{equation}
where $C>0$ is a constant independent of $u$ and $r$.
\medskip

Before giving our main theorems, we introduce some notations and preliminary facts. 
\smallskip

It is well-known that  the following problem 
\be\label{Ufangcheng}
\begin{cases}
-\Delta u + u = u^{p}, & \text{in } \mathbb{R}^{n}, \\ 
u > 0, & \text{in } \mathbb{R}^{n},
\end{cases}
\ee
with \(n \geq 1\), \(p>1\) and \(u \in H^{1}(\mathbb{R}^{n})\) admits a ground state solution $Q\in H^1(\R^n)$, which possesses the following properties
\begin{equation}\label{jitaijie}
\left\{
\begin{array}{ll} 
Q(x) = Q(|x|), & \text{for all } x \in \mathbb{R}^{n}, \\ 
Q(r) > 0, & \text{for all } r > 0, \\ 
\lim_{r \to \infty} e^{r} r^{\frac{n-1}{2}} Q(r) = \alpha_{n, p} > 0, & \lim_{r \to \infty} \frac{Q'(r)}{Q(r)} = -1,
\end{array}
\right.
\end{equation}
where \(\alpha_{n, p}\) is a constant depending only on \(n\) and \(p\). In addition, all solutions of \eqref{Ufangcheng} belonging to \(H^{1}(\mathbb{R}^n)\) can be expressed as \(Q(x - x_{0})\) for some \(x_{0} \in\mathbb{R}^n\).
\smallskip

In the case \(n=1\), the solutions to problem \eqref{Ufangcheng} correspond to the critical points of the functional \(\bar{J}_{0}: H^{1}(\mathbb{R}) \to \mathbb{R}\) given by
\begin{equation*}
\bar{J}_{0}(u) = \frac{1}{2} \int_{\mathbb{R}}\left(|\nabla u|^{2} + u^{2}\right) - \frac{1}{p+1} \int_{\mathbb{R}}|u|^{p+1}, \quad u \in H^{1}(\mathbb{R}).   
\end{equation*} 
The set 
\[
\mathcal Q_{0} = \left\{Q\left(x - x_{0}\right) : x_{0} \in \mathbb{R}\right\} \subseteq H^{1}(\mathbb{R})
\]
forms a manifold consisting of mountain pass type critical points for the functional \(\bar{J}_0\). The manifold \(\mathcal Q_{0}\) possesses the subsequent non-degeneracy property.
\medskip

\begin{specialproposition}
\label{prop2.111}
The manifold \(\mathcal Q_{0}\) exhibits non-degeneracy relative to the functional \(\bar{J}_{0}\). To be more precise, there exists some positive constant $C$ satisfying 
\[
\bar{J}_{0}''(Q)[Q, Q] \leq -C^{-1}\|Q\|_{0}^{2}, \quad \bar{J}_{0}''(Q)[v, v] \geq C^{-1}\|v\|_{0}^{2},
\]
for all \(v \in H^{1}(\mathbb{R})\), \(v \perp Q\), \(v \perp T_{Q} \mathcal Q_{0}\), where \(\|\cdot\|_{0}\) denotes the natural norm of \(H^{1}(\mathbb{R})\).
\end{specialproposition}
\medskip

Given any \(\lambda > 0\), we define \(Q_{\lambda}\) as the solution of
\begin{equation}\label{lambdadefangcheng}
\left\{
\begin{array}{l} 
-u'' + \lambda^{2} u = u^{p}, \quad \text{in } \mathbb{R}, \\[0.3em]
u'(0) = 0, \quad u > 0, \quad u \in H^{1}(\mathbb{R}).
\end{array}
\right.    
\end{equation}
Similarly, $Q_{\lambda}$ is a critical point of the functional
\[
\bar{J}_{\lambda}(u) 
= \frac{1}{2} \int_{\mathbb{R}}\big(|\nabla u|^{2} + \lambda^{2} u^{2}\big)\, dx
- \frac{1}{p+1} \int_{\mathbb{R}} |u|^{p+1}\, dx, 
\quad u \in H^{1}(\mathbb{R}),
\]
and the following scaling relation holds:
\[
Q_{\lambda}(s) = \lambda^{\tfrac{2}{p-1}} \, Q(\lambda s).
\]
From \eqref{jitaijie},  we know that $Q_{\lambda}$ is a radial function and decays exponentially. 
Precisely, there exists a constant $c>0$ such that 
\begin{equation}\label{shuaijianxing}
Q_{\lambda}(s), \; Q_{\lambda}'(s), \; Q_{\lambda}''(s) \;\leq\; c e^{-\lambda |s|}, \quad |s| > 1.
\end{equation}

For each $\rho>0$, let $U=U_{\rho,\varepsilon}(r)$ denote the solution of \eqref{lambdadefangcheng} with parameter
\[
\lambda^{2}=1+\varepsilon^{2}V(\varepsilon \rho),
\]
that is,
\begin{equation}\label{U_rhovar}
\left\{
\begin{aligned}
&-u'' + \big(1+\varepsilon^{2}V(\varepsilon \rho)\big) u = u^{p}, \quad \text{in } \mathbb{R},\\
&u'(0)=0,\qquad u>0,\qquad u\in H^{1}(\mathbb{R}).
\end{aligned}
\right.
\end{equation}

\smallskip

The purpose of this paper is twofold. 
\smallskip

First, we derive necessary conditions
for the occurrence of radial concentration for problem
\eqref{fangcheng1}--\eqref{fangcheng11} in the normalized setting, and we
provide a precise description of the limiting behavior of concentrating
solutions as the mass parameter $a\to a_{0}\in[0,+\infty]$. In particular,
we identify the possible concentration mechanisms, determine the location of
the concentrating spheres and the corresponding limiting process, and obtain
the asymptotic behavior of the associated Lagrange multiplier $\mu_a$.

Second, we show that these conditions are indeed realized by constructing
radial normalized solutions whose mass concentrates on spheres whose radii
tend to infinity, that is, on expanding spherical layers escaping to
infinity as $a\to+\infty$. The existence of such solutions is proved by a
finite-dimensional reduction.
\smallskip

To implement the construction through a finite-dimensional reduction argument, we look for radial solutions of the form
\begin{equation}\label{jiedexingshi1}
u_{\varepsilon}(r)
=U_{\rho,\varepsilon}\!\left(\frac{r}{\varepsilon}-\rho\right)
+\omega_{\rho,\varepsilon}\!\left(\frac{r}{\varepsilon}\right),
\end{equation}
where $\rho=\rho(\varepsilon)$ is an unknown concentrating radius and the
small parameter $\varepsilon>0$ is related to $a$ in such a way that
$\varepsilon\to0$ as $a\to a_{0}$. The remainder $\omega_{\rho,\varepsilon}$ is required be a higher-order
correction, namely
\begin{equation}\label{jiedexingshi2}
\Big\|\omega_{\rho,\varepsilon}\Big(\tfrac{r}{\varepsilon}\Big)\Big\|_\varepsilon
= o\!\left(\Big\|U_{\rho,\varepsilon}\Big(\tfrac{r}{\varepsilon}-\rho\Big)\Big\|_{\varepsilon}\right)
= o\big(\sqrt{\varepsilon^{n}\rho^{\,n-1}+\varepsilon^{n}}\big).
\end{equation}

\smallskip

Specifically, our  first  result concerns the value of $a_0$ and the location of the concentrating sphere, which is the  necessary conditions for concentration  by combining suitable Pohozaev-type identities with the constraint condition.

\smallskip

\begin{Thm}\label{thm0}
Suppose $V(x)=V(|x|)\in C^{1}(\mathbb{R}^{+}, \mathbb{R})$ and $V^\prime$ are bounded and $u_\varepsilon(r)$ is a concentrated solution of problem \eqref{fangcheng1}-\eqref{fangcheng11} of the form \eqref{jiedexingshi1}-\eqref{jiedexingshi2} as $a\rightarrow a_0$. Then $a_0=+\infty$ and $\rho\to+\infty$, $\e\rho\to+\infty$, $M_\e'(\e\rho)\rightarrow0$ or $V'(\e\rho)\rightarrow0$ as $a\rightarrow +\infty$.

\end{Thm}
\smallskip

\begin{remark}	Previous works \cite{Luo,Guo2,Guo3} on normalized solutions concentrating at isolated points show that the limiting value $a_{0}$ depends in a delicate way on the exponent $p$: in the mass subcritical case with $1<p<\frac4n+1$, one has $a_{0}=+\infty$, in the mass supercritical case with $p>\frac4n+1$, $a_{0}=0$, while in the mass critical case with $p=\frac4n+1$, $a_{0}$ is a positive constant determined by the ground state. Our results reveal a rather different picture for high-dimensional solutions concentrating on spheres. We prove that, for every $p>1$, the limiting process is independent of the range of $p$ and the concentration radius of the sphere tends to infinity. This unifies and extends the mass critical result in \cite{gtzcv} for the special two-dimensional case $n=2$, $p=3$. At the same time, we will point out later that the two-dimensional arguments in \cite{gtzcv} do not directly extend to higher dimensions (see Remark \ref{rem3.4}); in our approach, both the approximate profile and the working space have to be modified.
\end{remark}	
\medskip

The second main result of this paper concerns the existence of concentrated solutions; in this case we need to impose additional assumptions on the potential $V$. 
\smallskip

Denote
\be\label{Mepr}
M_\e(r) := \e^{2(n-2)}r^{n-1}\left(1+\e^2V(r)\right)^{\frac{p+3}{2(p-1)}}.
\ee  Suppose that $M_{\e}$ has a family of non-degenerate critical
points which tends to inﬁnity. Precisely, we assume that for any $\e$ small enough, there exist a sequence $\{t_\e\}\subset \R^+$ such that there exist positive constants $\beta>0$ and $\e_0>0$ such that
\begin{equation}\label{Mdelinjiedian}
M_{\e}'\left(t_{\e}\right)=0,   \ \ |M''(t_\e)|\geq \beta \ for\ all\ \e\in(0,\e_0).
\end{equation}
Without loss of generality, we also assume for some positive constants $C_1$ and $C_2$,
\begin{equation}\label{linjiediandejieshu}
C_1\e^{-2}\leq t_\e\leq C_2\e^{-2}.  
\end{equation}
\smallskip

\begin{Thm}
\label{thm1.1}
Suppose $V(x)=V(|x|)\in C^{1}(\mathbb{R}^{+}, \mathbb{R})$ and $V^{\prime}$ are bounded , while $M_\e(r)$ meets \eqref{Mdelinjiedian} and $\eqref{linjiediandejieshu}$. For any $p > 1$, as $a\to+\infty$, there exist a sequence radial solutions $\{u_{\e_a}(r)\}$ to problem \eqref{fangcheng1}-\eqref{fangcheng11} with the form \eqref{jiedexingshi1}-\eqref{jiedexingshi2} which concentrate near the sphere $|x| = t_{\e_a}$.  
\end{Thm} 
\smallskip

\begin{Rem}Arguing similarly as in \cite{gtzcv}, we can also construct an example
satisfying the condition \eqref{Mdelinjiedian} and \eqref{linjiediandejieshu} such as $V(r)=\sin r$.

\end{Rem}
\iffalse
\begin{Rem}
In the proof of Theorem \ref{thm1.1}, we first treat the case $1<p \leq \frac{n+2}{n-2}$. 
When $p > \frac{n+2}{n-2}$, the argument is adapted by introducing a truncation of the nonlinear term, 
which then allows us to derive a priori $L^\infty$-estimates for the solutions. 
The detailed proof of this case is presented in Subsection \ref{5.1}.

\end{Rem}
\fi

\medskip

The structure of the paper is as follows. In Section~2, we first derive
necessary conditions for the existence of concentrated solutions to
problem \eqref{fangcheng1}–\eqref{fangcheng11}, which characterize the
precise limiting process and the location of the concentration points.
In Section~3, we begin the construction of concentrated solutions by the
reduction method. We first study the unconstrained problem with the
free parameter $\varepsilon \to 0$ and, by means of a fixed point
argument, obtain solutions in the orthogonal complement of the
approximate kernel, thereby reducing the problem to a one-dimensional
problem. Section~4 is devoted to proving the existence of solutions to
the reduced problem. Finally, using the intermediate value theorem and
further analytic estimates, we obtain a normalized solution satisfying
the constraint; reverting the change of variables to the original
setting, we conclude the existence of a solution $u_a$ to the original
problem together with the associated Lagrange multiplier $\mu_a$.

\medskip

\section{
	limiting behavior}
In this section, we  apply various   Pohozaev-type identities to obtain essential estimates for $u_\e(r)$ and the error term $\omega_{\rho,\e}$. Then the asymptotic behavior and the proof of Theorem \ref{thm0} are given by the following Lemma \ref{Propa}, Lemma \ref{Propa11}, Lemma \ref{Lemma2.10} and Lemma \ref {lemma2.11}.
\medskip

Recall that, throughout this paper,  $Q\in H^1(\R)$ denotes the unique positive solution to problem \eqref{Ufangcheng} and $u_\varepsilon(r)$ is a radial function of the form  \eqref{jiedexingshi1}-\eqref{jiedexingshi2}, i.e.,
\begin{equation*}
u_{\varepsilon}(r)=U_{\rho,\varepsilon}\left(\frac{r}{\e}-\rho\right)+\omega_{\rho,\varepsilon}\left(\frac{r}{\e}\right),
\end{equation*}
where $$\left\|\omega_{\rho,\varepsilon}\left(\frac{r}{\e}\right)\right\|_\varepsilon=o\left(\left\|U_{\rho,\varepsilon}\left(\frac{r}{\e}-\rho\right)\right\|_{\e}\right)=o\left(\sqrt{\e^n\rho^{n-1}+\e^n}\right).$$
From definition \eqref{U_rhovar}, we obtain that
\be\label{zhuxiangfangcheng}
-\e^2\zhuxianga^{\prime\prime}+\ber^2\zhuxiang=\zhuxiang^p,\ee
where $\beta_{\e,\rho}:=\left(1+\e^2V(\e\rho)\right)^{\frac{1}{2}}.$	

\smallskip

First, we have the identity for $Q$ by use of the Pohozaev identities standardly.
\begin{Lem}
The following holds
\begin{equation}\label{pohU}
\frac{p-1}{2(p+1)}\int^{+\infty}_0 Q^{p+1}(r)dr=\frac{p-1}{p+3}\int^{+\infty}_0Q^2(r)dr=\int^{+\infty}_0 \big|Q'(r)\big|^2dr.
\end{equation}
%where $U\in H^1(\R)$ is the positive solution of  \eqref{8-19-1}.
\end{Lem}
~\\	
\smallskip

By standard arguments (see, for instance, the proofs of Lemmas~2.2 and~2.3
in \cite{gtzcv} or the related discussion in \cite{Gilbarg}), we can derive
a priori estimates on $u_{\varepsilon}(x)$ and $\yuxianga$.

\begin{Lem}
Suppose that $u_\varepsilon(x)$ is a positive radial solution of \eqref{fangcheng1}   with \eqref{jiedexingshi1}-\eqref{jiedexingshi2}. Then
for any fixed $R\gg 1$, there exist $\eta>0$ and $C>0$, such that
\begin{flalign}\label{2--5}
u_\varepsilon(x)\leq Ce^{-\eta|x-x_\varepsilon|/\varepsilon},~\mbox{for any}~ x\in \R^n ~\mbox{with}~|x_\varepsilon|=\e\rho.
\end{flalign}
\end{Lem}

\smallskip

%Using above Lemma \ref{prop1-1}, we have following estimate on the error term $w_\e(r)$.
\begin{Lem}
There exists $\eta\in (0,1)$ such that
\begin{flalign*}
\omega_{\rho,\varepsilon}\left(\frac{r}{\e}\right)=O\Big(e^{-\eta|r-\e\rho|/\varepsilon}\Big),~\mbox{for}~ r\geq 0.
\end{flalign*}
\end{Lem}
\medskip

Note that if $u_\e$ is radial,  %it holds that $\Delta u=u''+\frac{N-1}{r}u'$. Thus,
\eqref{fangcheng1} can be rewritten as
\begin{equation}\label{eq3}
-\varepsilon^2u_\e''(r)-\frac{(n-1)\varepsilon^2 }{r} u_\e'(r)+\left(1+\varepsilon^2V(r)\right)u_\e(r)
=u_\e^p(r).
\end{equation}
\medskip
Then we have following  identity on $u_\e$ by multiplying $r^{n-1}u_\varepsilon$ on both sides of \eqref{eq3} and integrating from $0$ to $+\infty$.
\begin{Prop}
If $u_\varepsilon$ is a radial  solution to \eqref{fangcheng1}, then we establish that
\begin{equation}\label{poh1}
\varepsilon^2\int^{+\infty}_0r^{n-1} \big|u'_\varepsilon\left(r\right)|^2dr+
\int^{+\infty}_0r^{n-1} \left(1+\varepsilon^2V(r)\right)u^2_\varepsilon (r) dr=\int^{+\infty}_0r^{n-1} u^{p+1}_\varepsilon(r) dr.
\end{equation}
\end{Prop}

\medskip
Now, we estimate each term in \eqref{poh1}.
\begin{Prop}
Let $u_\varepsilon$ be a radial solution of \eqref{fangcheng1} with the form \eqref{jiedexingshi1}-\eqref{jiedexingshi2}, then it holds that
\begin{equation}\label{ll1}\begin{split}
\varepsilon^2\int^{+\infty}_0r^{n-1} \left|u'_\varepsilon(r)\right|^2dr=&
\varepsilon^2\int^{+\infty}_0  r^{n-1} \left|\zhuxianga^\prime\right|^2dr\\&+ O\left(\wucha\left\|\yuxianga\right\|_\e+\left\|\yuxianga\right\|_\e^2\right),		    
\end{split}
\end{equation}
\begin{equation}\label{ll2}\begin{split}
\int^{+\infty}_0r^{n-1} \left(1+\varepsilon^2V(r)\right)u^2_\varepsilon (r) dr=&\int^{+\infty}_0r^{n-1} \big(1+\varepsilon^2V(r)\big)\zhuxianga^2dr\\&+ O\left(\wucha\left\|\yuxianga\right\|_\e+\left\|\yuxianga\right\|_\e^2\right), 
\end{split}
\end{equation}
and		\begin{equation}
\label{ll3}\begin{split}
\int^{+\infty}_0r^{n-1} u^{p+1}_\varepsilon(r) dr=&\int^{+\infty}_0r^{n-1} \zhuxianga^{p+1}dr
\\&+ O\left(\wucha\left\|\yuxianga\right\|_\e+\left\|\yuxianga\right\|_\e^2\right).\end{split}
\end{equation}
\end{Prop}

\begin{proof}
First, it is easy to find that
\begin{equation}\label{ll1d}
\begin{split}
\varepsilon^2\int^{+\infty}_0r^{n-1} \big|u'_\varepsilon\left(r\right)|^2dr=&\varepsilon^2\int^{+\infty}_0r^{n-1} \left|\zhuxianga^\prime\right|^2dr+\varepsilon^2\int^{+\infty}_0 r^{n-1}{\left(\yuxianga\right)^\prime}^2 dr\\&+2\varepsilon^2\int^{+\infty}_0 r^{n-1} \zhuxianga^\prime\left(\yuxianga\right)^\prime dr
.
\end{split}\end{equation}
Applying H\"older's inequality, we derive that
\begin{equation}\label{ll1da}\begin{split}
\varepsilon^2\int^{+\infty}_0 r^{n-1} \zhuxianga^\prime\left(\yuxianga\right)^\prime dr
&=O\Big(\left\|\zhuxiang\right\|_\varepsilon\cdot
\left\|\yuxianga\right\|_\varepsilon\Big)\\
&=O\left(\wucha\left\|\yuxianga\right\|_\varepsilon\right),  
\end{split}
\end{equation}
and \begin{equation}\label{8-26-2}
\varepsilon^2\int^{+\infty}_0 r^{n-1}{\left(\yuxianga\right)^\prime}^2 dr=O\Big(\left\|\yuxianga\right\|_\varepsilon^2\Big).
\end{equation}
Thus, \eqref{ll1} is a direct consequence of \eqref{ll1d}, \eqref{ll1da} and \eqref{8-26-2}.

\vskip 0.1cm

Analogous to the computations for \eqref{ll1}, we obtain that
\begin{equation*}\begin{split}
\int^{+\infty}_0r^{n-1} \left(1+\varepsilon^2V(r)\right)u^2_\varepsilon (r) dr=&\int^{+\infty}_0r^{n-1} \big(1+\varepsilon^2V(r)\big)\zhuxianga^2dr\\&+ O\left(\wucha\left\|\yuxianga\right\|_\e+\left\|\yuxianga\right\|_\e^2\right). 
\end{split}
\end{equation*}

Now for $\displaystyle\int^{+\infty}_0r^{n-1} u^{p+1}_\varepsilon(r) dr$, it holds then
\begin{equation}\label{llh1}
\begin{split}
\int^{+\infty}_0r^{n-1} u^{p+1}_\varepsilon(r) dr=&\int^{+\infty}_0r^{n-1} \zhuxianga^{p+1}dr\\&+(p+1)\int^{+\infty}_0 r^{n-1}\zhuxianga^p\yuxianga dr\\&
+
O\left(\int^{+\infty}_0 r^{n-1}\left(\zhuxianga^{p-1}\yuxianga^2+\left|\yuxianga\right|^{p+1} \right)dr\right).
\end{split}
\end{equation}
Since $\zhuxiang$ is bounded, we have
\begin{equation}\label{llh2}\begin{split}
\int^{+\infty}_0 r^{n-1}\zhuxianga^p\yuxianga dr&= O\left(\left\|\zhuxiang\right\|_\varepsilon\cdot
\left\|\yuxianga\right\|_\varepsilon\right)\\&=O\left(\wucha\left\|\yuxianga\right\|_\varepsilon\right).\end{split}
\end{equation}
It follows from \eqref{2--5} that
$\zhuxiang$ and  $\yuxianga$ are uniformly bounded for any $r\geq 0$,
\begin{equation}\label{llh3}
\begin{split}
\int^{+\infty}_0 r^{n-1}\left(\zhuxianga^{p-1}\yuxianga^2+\left|\yuxianga\right|^{p+1} \right)dr=O\left(\left\|\yuxianga\right\|_\e^2\right).
\end{split}\end{equation}
Finally, \eqref{ll3} is a direct consequence of \eqref{llh1}, \eqref{llh2} and \eqref{llh3}.
\end{proof}
\medskip

From the above analysis, we get an elementary necessary condition as follows, which is an important constitution of Theorem \ref{thm0}.
\begin{Lem}\label{Propa}
Let $u_\varepsilon(r)$ be a radial solution of \eqref{fangcheng1} with the form \eqref{jiedexingshi1}-\eqref{jiedexingshi2}.
There holds that
\begin{equation}\label{r/epsilon}
\rho\rightarrow +\infty,~\mbox{as}~\varepsilon\rightarrow 0.
\end{equation}

%Moreover, if  $u_\varepsilon$ satisfies the constraint \eqref{fangcheng1'}, then  we have the following behavior for $a$:
%\smallskip
%
%%(a1) %{\bf When $N-3$ is even or $N=2$ or  when $N-3$ is odd and \eqref{r/epsilon} holds,}
%%Under \eqref{r/epsilon},   the following holds:
%When $p-1\leq\frac4N$,  $a\rightarrow+\infty$;
%
%When $p-1>\frac4N$, we have three cases.
%(i) If $ \frac{r_\varepsilon^{n-1}}{\varepsilon^{\frac4{p-1}-1}}\rightarrow0$,   then   $a\rightarrow0$.
%(ii) If $ \frac{r_\varepsilon^{n-1}}{\varepsilon^{\frac4{p-1}-1}}\rightarrow+\infty$,    then  $a\rightarrow+\infty$.
%(iii) If there exists some $c_0>0$ such that $ \frac{r_\varepsilon^{n-1}}{\varepsilon^{\frac4{p-1}-1}}\rightarrow c_0$,   then $a\rightarrow A_0$
%with $$A_0=\Big(2\omega_{N-1}c_0\int_0^\infty U^2(t)dt\Big)^{\frac{p-1}2}.$$
%
%%(a2) %{\bf When  $N-3$ is odd and }
%%In the case of \eqref{r/epsilon0},   there holds the following:
%%
%%If $p-1>\frac4N$,   then $a\rightarrow0$;
%%
%%If $p-1<\frac4N$,  then $a\rightarrow+\infty$;
%%
%%If critically $p-1=\frac4N$, then $a\rightarrow A_1$ with
%% $$A_1=\Big(\omega_{N-1}c_0\int_0^\infty  t^{n-1} U^2(t)dt\Big)^{\frac2{N}}.$$

\end{Lem}

\begin{proof}
From  \eqref{poh1}, \eqref{ll1}, \eqref{ll2} and \eqref{ll3}, we infer
\begin{align*}
\begin{split}
&\varepsilon^2\int^{+\infty}_0 r^{n-1} \left|\zhuxianga^\prime\right|^2dr+
\int^{+\infty}_0r^{n-1} \big(1+\varepsilon^2V(r)\big)\zhuxianga^2dr\\
=&\int^{+\infty}_0r^{n-1} \zhuxianga^{p+1}dr
+o\big(\e^n\rho^{n-1}+\e^n\big),
\end{split}\end{align*}
which and \eqref{zhuxiangfangcheng} imply that
\begin{equation}\label{rN-2}
\begin{split}
\varepsilon^2\int^{+\infty}_0r^{n-2} \zhuxianga\zhuxianga^\prime dr
=o\left(\e^n\rho^{n-1}+\e^n\right).\end{split}\end{equation}

(1) When $n=2$, it follows from \eqref{rN-2} that
\begin{equation}\label{8-14-10}
\begin{split}
U_{\rho,\e}^2\left(\rho\right)=o\left(\rho+1\right).
\end{split}\end{equation}
If by contradiction $\rho\leq C_0$, then
\eqref{8-14-10} gives us that, for some positive constant $C$, $Q^2(C)
=0$, which is impossible. Hence we deduce \eqref{r/epsilon}.

(2) When $n\geq3$, from \eqref{rN-2} we get then
\begin{equation}\label{rN-2'}
\begin{split}
&\varepsilon^2\int^{+\infty}_0r^{n-2} \zhuxianga\zhuxianga^\prime dr\\
=&-\frac{(n-2)\varepsilon^3}{2}\beta_{\e,\rho}^{\frac{4}{p-1}-1}\int^{+\infty}_{(-\rho)\beta_{\e,\rho}}
\left(\frac{\varepsilon t}{\beta_{\e,\rho}}+\e\rho\right)^{n-3}Q^2(t)dt
\\
=&-\frac{(n-2)\varepsilon^3}{2}\beta_{\e,\rho}^{\frac{4}{p-1}-1}\int^{+\infty}_{(-\rho)\beta_{\e,\rho}}\sum_{i=0}^{n-3}{n-3\choose i}\left(\frac{\e t}{\beta_{\e,\rho}}\right)^i(\e\rho)^{n-3-i}
Q^2(t)dt\\
=&o\left(\e^n\rho^{n-1}+\e^n\right),
\end{split}
\end{equation}
where $\beta_{\e,\rho}:=\left(1+\e^2V(\e\rho)\right)^{\frac{1}{2}}$.
Obviously, \eqref{rN-2'} implies that there exists some constants $a_i,i=0,1,\ldots,n-3$, such that
$$(-1)\sum_{i=0}^{n-3}a_i\rho^{n-3-i}=o(\rho^{n-1}+1),$$
with
$$a_{i}:={n-3\choose i}\int_{(-\rho)\beta_{\e,\rho}}^{+\infty}Q^{2}(t)t^{i}dt,\ \ i=0,1,\dots,n-3.$$
We note that, if $i$ is even, then $a_i\geq C_i$. On the other hand, if $i$ is odd, we get
$$a_{i}\left\{
\begin{array}{ll}
\to 0,~~~\hbox{if}~~\rho\to +\infty,\\
\geq C_i^{\prime},~~~\hbox{otherwise.}\\
\end{array}
\right.$$
In the above, $C_i$ and $C_i^\prime$ are  positive constants independent of $\e$. This rules out the possibility of $\rho\leq C_0$ and concludes \eqref{r/epsilon}.
\end{proof}
\medskip

Now, we establish another Pohozaev identity on $u_\e(r)$.
\begin{Prop}
If $u_\varepsilon$ be a radial solution of \eqref{fangcheng1}, then it holds that
\begin{equation}\label{a8-10-6}
\varepsilon^2\int^{+\infty}_0r^{n-1} \big|u'_\varepsilon(r)\big|^2dr-\frac{\varepsilon^2}{2}
\int^{+\infty}_0r^{n}  V'(r)u^2_\varepsilon (r)dr= n\left(\frac12-\frac1{p+1}\right)\int^{+\infty}_0r^{n-1}  u^{p+1}_\varepsilon (r)dr.
\end{equation}
\end{Prop}

\begin{proof}
Multiplying $r^{n}u'_\varepsilon$ on both sides of \eqref{eq3} and integrating from $0$ to $+\infty$, we get
\begin{equation*}
\int^{+\infty}_0r^{n}\left(-\varepsilon^2\left(u''_\varepsilon+\frac{n-1}{r} u'_\varepsilon\right)
+\left(1+\varepsilon^2V(r)\right)u_\varepsilon
-u_\varepsilon^p\right)u'_{\varepsilon}dr=0,
\end{equation*}
which gives
\begin{align}\label{b8-10-6}
\begin{split}
	&\frac{\varepsilon^2(n-2)}{2}\int^{+\infty}_0r^{n-1} \big|u'_\varepsilon(r)\big|^2dr+\frac{\varepsilon^2}2\int^{+\infty}_0r^{n}V'(r)u^2_\varepsilon dr+\frac n2
	\int^{+\infty}_0r^{n-1} \big(1+\varepsilon^2V(r)\big)u^2_\varepsilon dr\\
	=&
	\frac n{p+1}\int^{+\infty}_0r^{n-1}  u^{p+1}_\varepsilon dr.
\end{split}\end{align}
Then combining
\eqref{poh1} and \eqref{b8-10-6}, we find \eqref{a8-10-6}.
\end{proof}

\medskip
%\subsection{ The case of $\frac{r_\varepsilon}{\varepsilon}\rightarrow+\infty$}

We next conduct a more precise estimate of each term appearing in \eqref{a8-10-6}.
\begin{Prop}
Let $u_\varepsilon(r)$ be a radial solution of \eqref{fangcheng1} with the form \eqref{jiedexingshi1}-\eqref{jiedexingshi2}, then it follows that
\begin{equation}\label{a8-10-3}
\begin{split}
	\varepsilon^2\int^{+\infty}_0r^{n-1} \big|u'_{\varepsilon}(r)\big|^2dr=&
	2A\ber^{\frac{p+3}{p-1}}\e^n\rho^{n-1}+o(\e^n\rho^{n-1}),
\end{split}\end{equation}
\begin{equation}\label{8-10-4}
\begin{split}
	\frac{\varepsilon^2}{2}\int^{+\infty}_0r^{n}V'(r)u_{\varepsilon}^2(r)dr=&
	\frac{p+3}{p-1}A\ber^{\frac{4}{p-1}-1}\e^{3+n}\rho^nV^\prime(\e\rho)+o(\e^{3+n}\rho^n),
\end{split}\end{equation}
\begin{equation}\label{8-10-422}
\begin{split}
	\int^{+\infty}_0&r^{n-1}\big(1+\varepsilon^2V(r)\big)  u^2_{\varepsilon}(r)dr=
	\frac{2(p+3)}{p-1} A\varepsilon^n\rho^{n-1} \ber^{\frac{4}{p-1}-1}+o\big(\e^n\rho^{n-1}\big),
\end{split}\end{equation}

and
\begin{equation}\label{8-14-5}
\begin{split}
	\int^{+\infty}_0r^{n-1} u^{p+1}_\varepsilon(r) dr=\frac{4(p+1)}{p-1} A\ber^{\frac{p+3}{p-1}}\e^n\rho^{n-1}+o\big(\e^n\rho^{n-1}\big),
\end{split}\end{equation}
where  $A:=\displaystyle\int^{+\infty}_0 |Q'(r)|^2dr$.
\end{Prop}

\begin{proof}Combining with $\rho\to+\infty$, 
we estimate that
\begin{equation}\label{8-26-1}
\begin{split}
&\varepsilon^2\int^{+\infty}_0r^{n-1}\left|\zhuxianga^\prime\right|^2dr\\=&
	\varepsilon \ber^{\frac{p+3}{p-1}} \int^{+\infty}_{(-\rho)\ber}\left(\e\rho+\frac{\varepsilon t}{\ber}\right)^{n-1}\big|Q'(t)\big|^2dt\\=&\displaystyle\sum\limits_{k=0}^{n-1}{n-1\choose k}(\e\rho)^{n-1-k}
	\varepsilon^{k+1} \ber^{\frac{p+3}{p-1}-k} \int^{+\infty}_{(-\rho)\ber}
	t^k\big|Q'(t)\big|^2dt
	\\=&
	2A \ber^{\frac{p+3}{p-1}}\e^n\rho^{n-1}+o\left( \e^n\rho^{n-1}\right).
\end{split}\end{equation}
Therefore \eqref{a8-10-3} follows directly from \eqref{jiedexingshi1}-\eqref{jiedexingshi2}, \eqref{ll1} and \eqref{8-26-1}.

\vskip 0.1cm

Following the same computation as in \eqref{8-26-1}, we obtain the subsequent estimate using \eqref{pohU},
\begin{equation}\label{8-26-4}
\begin{split}
	&\frac{\e^2}{2}\int^{+\infty}_0r^{n}V'(r) \zhuxianga^2dr\\=&\frac{\e^2}{2}\int^{+\infty}_0r^{n}V'(\e\rho) \zhuxianga^2dr+\frac{\e^2}{2}\int^{+\infty}_0r^{n}\left(V'(r)-V^\prime(\e\rho)\right) \zhuxianga^2dr\\=&\frac{p+3}{p-1}A\e^{n+3}\rho^n\ber^{\frac{4}{p-1}-1}V^\prime(\e\rho)+o(\e^{n+3}\rho^n)
	.
\end{split}
\end{equation}
Then by H\"older's inequality, we have
\begin{equation}\label{8-26-5}
\begin{split}
	&\int^{+\infty}_0r^{n}V'(r) \zhuxiang\yuxianga dr \\=&
	O\left(\left(\int^{+\infty}_0r^{n+1}\big|V'(r) \big|^2 \zhuxianga^2dr\right)^{\frac{1}{2}}\left\|\yuxianga\right\|_\varepsilon\right) \\
	=&O\left(\left(\e\rho+\varepsilon\right)^{\frac{n+1}{2}} \varepsilon^{\frac{1}{2}}\left\|\yuxianga\right\|_\varepsilon\right),
\end{split}
\end{equation}
and
\begin{equation}\label{8-26-6}
\begin{split}
	\int^{+\infty}_0r^{n}V'(r) \yuxianga^2dr=
	O\Big(\big(\e\rho+ \varepsilon\big)\left\|\yuxianga\right\|^2_\varepsilon\Big).
\end{split}
\end{equation}
Thus, \eqref{8-26-4}, \eqref{8-26-5} and \eqref{8-26-6} conjointly lead to \eqref{8-10-4}.

\vskip 0.1cm

Finally, we calculate
\begin{equation}\label{8-26-411}
\begin{split}
	\int^{+\infty}_0r^{n-1} (1+\e^2V(r))\zhuxianga^2 dr=\frac{2(p+3)}{p-1} A\varepsilon^n\rho^{n-1} \ber^{\frac{4}{p-1}+1}+o\big(\e^n\rho^{n-1}\big),
\end{split}
\end{equation}
and
\begin{equation}\label{a8-26-7}
\begin{split}
	\int^{+\infty}_0r^{n-1}\zhuxianga^{p+1}dr=
	\frac{4(p+1)}{p-1} A\ber^{\frac{p+3}{p-1}}\e^n\rho^{n-1}+o\big(\e^n\rho^{n-1}\big).
\end{split}
\end{equation}
Analogously, we deduce \eqref{8-10-422} from the combination of \eqref{pohU}, \eqref{ll2}, and \eqref{8-26-411}. Additionally, \eqref{8-14-5} holds by \eqref{pohU},  \eqref{ll3} and \eqref{a8-26-7}.
\end{proof}
\smallskip

\begin{Lem}\label{Propa11}
Let $u_\varepsilon(r)$ be a radial solution of \eqref{fangcheng1} with \eqref{jiedexingshi1}-\eqref{jiedexingshi2}.
There holds that
\begin{equation*}
\e\rho\rightarrow +\infty,~\mbox{as}~\varepsilon\rightarrow 0.
\end{equation*}
\end{Lem}

\begin{proof}

From \eqref{a8-10-6}, \eqref{a8-10-3}, \eqref{8-10-4} and \eqref{8-14-5}, we infer
\begin{equation}\label{9-07-3}
2(n-1)\er^{\frac{p+3}{2(p-1)}}+\frac{p+3}{p-1}\er^{\frac{2}{p-1}-\frac{1}{2}}\e^3\rho V^\prime(\e\rho)=o\left(\e^3\rho\right)+o(1).
\end{equation}
Suppose that $\e\rho$ is bounded, then \eqref{9-07-3} means $2(n-1)=o(1),$ which is impossible.

Hence, it is inferred that $$\e\rho \rightarrow +\infty,~\mbox{as}~\varepsilon\rightarrow 0.$$
\end{proof}
\begin{Rem}
Adopting the same approach as in Lemma \ref{Propa11}, we further deduce that for any \(0<\alpha<3\), as \(\varepsilon\to0\),\begin{equation}\label{alphajixian1}\varepsilon^\alpha\rho\to+\infty,\end{equation} that will be used in the subsequent Lemma \ref{Lemma2.10}.
\end{Rem}

\smallskip

Under the $L^{2}$-constraint, $a$ has the following asymptotic behavior.

\begin{Lem}\label{Lemma2.10}
Let $u_\varepsilon(r)$ be a radial solution of \eqref{fangcheng1} with the form \eqref{jiedexingshi1}-\eqref{jiedexingshi2} satisfying the constraint \eqref{fangcheng11}, then we have, as $\e\to0$, $a\to+\infty$.
\end{Lem}
\begin{proof}
From \eqref{fangcheng11} and \eqref{ll3}, we have
\begin{align}\label{a1}
\begin{split}
	\frac{a^{\frac{2}{p-1}}\e^{\frac{4}{p-1}}}{\omega_{n-1}}&=\int^{+\infty}_0r^{n-1} u_\e^2(r)dr\\
&=\int_0^{+\infty}r^{n-1}\zhuxianga^2dr+o\left(\e^n\rho^{n-1}+\e^n\right).
\end{split}
\end{align}

In addition, since $$U_{\rho,\e}(r)=(1+\e^2V(\e\rho))^{\frac{1}{p-1}}Q((1+\e^2V(\e\rho))^{\frac{1}{2}}r),$$ it is easy to check that
\begin{equation}
\label{diyibufen}
\begin{split}
\int_0^{+\infty}r^{n-1}\zhuxianga^2dr=&2\frac{p+3}{p-1}A\e^{n}\rho^{n-1}\ber^{\frac{4}{p-1}-1}+o(\e^n\rho^{n-1})\\
=&\e^n\rho^{n-1}C_0(1+o(1)),
\end{split}
\end{equation}
where $C_0=2\frac{p+3}{p-1}A$.
Putting \eqref{a1} and \eqref{diyibufen} together, we obtain
\begin{equation*}
\frac{a^{\frac{2}{p-1}}\e^{\frac{4}{p-1}}}{\omega_{n-1}}=C_0\e^n\rho^{n-1}(1+o(1))+o\left(\e^n\rho^{n-1}+\e^n\right),
\end{equation*}
which turns out that
\begin{equation}\label{eq9191}
a^{\frac{2}{p-1}}=C_0\omega_{n-1}(\e\rho)^{n-1}\e^{1-\frac{4}{p-1}}(1+o(1)).
\end{equation}
\smallskip

First, under the assumption $1<p\le 5$ we have $1-\frac{4}{p-1}\le 0$, which,
combined with \eqref{eq9191} and Lemma~\ref{Propa11}, implies that
\[
a \to +\infty \quad \text{as } \varepsilon \to 0.
\]

If $p>5$, \eqref{eq9191} implies that, 
\begin{equation*}
a^\frac{2}{p-1}=C_0\omega_{n-1}\left(1+o(1)\right)
\left(\rho\varepsilon^{\frac{1}{n-1}\left({1-\frac{4}{p-1}}\right)+1}\right)^{n-1}.  
\end{equation*}

On the other hand, we have obviously that
$$1<\frac{1}{n-1}\left({1-\frac{4}{p-1}}\right)+1<2.$$
Considering  \eqref{alphajixian1},  we obtain that, as $\e\to0$,
$$\rho\e^{\frac{1}{n-1}\left(1-\frac{4}{p-1}\right)+1}\to+\infty,$$
which consequently leads to the conclusion that $a\to+\infty$.
Thereby, we establish Lemma \ref{Lemma2.10}.
\end{proof}
\smallskip

\begin{Lem}\label{lemma2.11}
Let  $M_\varepsilon(r)=\e^{2(n-2)}r ^{n-1}\big(1 + \varepsilon^2V(r)\big)^{\frac{p+3}{2(p-1)}}$. If  $u_\varepsilon(r)$ be  a radial solution of \eqref{fangcheng1} with the form \eqref{jiedexingshi1}-\eqref{jiedexingshi2}, then we derive
$$M_{\varepsilon}'(\e\rho)=o(1)~~\hbox{or}~~V'(\e\rho)=o(1).$$
\end{Lem}
\begin{proof}
By virtue of \eqref{9-07-3}, we deduce that, as $\varepsilon \to0$,
\begin{equation}\label{9192}
\begin{split}
M'_\varepsilon(\e\rho)=&\frac12\left(\e\rho\right)^{n-2}\e^{2(n-2)}\left(2\left(n-1\right)\er^{\frac{p+3}{2(p-1)}}+\frac{p+3}{p-1}\varepsilon^3\rho\er^{\frac{2}{p-1}-\frac{1}{2}}V'(\e\rho) \right)\\=&o\left(\left(\e^{3}\rho\right) ^{n-2}\right)+o\left(\left(\e^{3}\rho\right) ^{n-1}\right).
\end{split}
\end{equation}
Consequently, if $\e^3\rho$ is bounded, \eqref{9192} gives $M_{\varepsilon}'(\e\rho)=o(1)$. While if $\e^3\rho\rightarrow+\infty$,  then \eqref{9-07-3} gives  $V'(\e\rho)=o(1)$.
\end{proof}

Accordingly, the proof of Theorem \ref{thm0} is concluded.

\medskip

\section{reduction for existence}

\medskip

\subsection{Ansatz}
To prove Theorem \ref{thm1.1} by the reduction method, we perform the change of variables
\[
\tilde{u}_\varepsilon(x) := u_\varepsilon(\varepsilon x),
\]
and rewrite the problem \eqref{fangcheng1}–\eqref{fangcheng11} as

\begin{equation}\label{disanjiediyige}
- \Delta \tilde{u}_\varepsilon +(1+\varepsilon^2 V(\varepsilon|x|)) \tilde{u}_\varepsilon  
= \tilde{u}_\varepsilon^{p}, \quad \tilde{u}_\varepsilon \in H^{1}(\mathbb{R}^{n}),
\end{equation}
subject to the constraint
\begin{equation}\label{disanjiedierge}
\int_{\mathbb{R}^n}\tilde{u}_\varepsilon^2(x)\, dx 
= a^{\tfrac{2}{p-1}} \varepsilon^{\tfrac{4}{p-1}-n}.
\end{equation}
\smallskip

From now on, we shall work with the rescaled unknown $\tilde{u}_\varepsilon(x)$
when proving Theorem~\ref{thm1.1}.

\smallskip

Solutions to \eqref{disanjiediyige} correspond to critical points of the \(C^2\) functional \(J_{\varepsilon}: H_{r}^{1} \to \mathbb{R}\), 
\[
\begin{aligned} 
J_{\varepsilon}(u) & = \frac{1}{2} \int_{\mathbb{R}^{n}}\left[|\nabla u|^{2} + (1+\e^2V(\varepsilon |x|) )u^{2}\right] dx - \frac{1}{p+1} \int_{\mathbb{R}^{n}}|u|^{p+1} dx \\ 
& = \frac{1}{2} \int_{0}^{+\infty} r^{n-1}\left[(u')^{2} + (1+\e^2V(\varepsilon r) )u^{2}\right] dr - \frac{1}{p+1} \int_{0}^{+\infty} r^{n-1}|u|^{p+1} dr.
\end{aligned}
\]
By replacing $u$ with its positive part $u_{+}$, we see that any critical
point $u$ of $J_{\varepsilon}$ must satisfy $u \ge 0$. Hence, by the maximum
principle, $u_{+}$ is in fact strictly positive.
\smallskip

Let $\zeta_{\varepsilon}\in C^{\infty}([0,\infty))$ be a nondecreasing cut-off function such that
\[
\zeta_{\varepsilon}(r)=
\begin{cases}
0, & r \le \tfrac{C_1}{16\,\varepsilon^{3}}, \\[0.3em]
1, & r \ge \tfrac{C_1}{8\,\varepsilon^{3}},
\end{cases}
\]
and
\[
|\zeta'_{\varepsilon}(r)| \le \frac{C_3\,\varepsilon^{3}}{C_1}, 
\qquad 
|\zeta''_{\varepsilon}(r)| \le \frac{C_4\,\varepsilon^{6}}{C_1^{2}}
\quad \text{for all } r\ge0.
\]
Moreover,
\[
\operatorname{supp}\zeta'_{\varepsilon}\,\cup\, \operatorname{supp}\zeta''_{\varepsilon}
\subset \Big[\tfrac{C_1}{16\,\varepsilon^{3}},\, \tfrac{C_1}{8\,\varepsilon^{3}}\Big].
\]
In what follows, we fix the configuration set $\Omega_\e = \left[\frac{C_1}{2\e^3}, \frac{2C_2}{\e^3}\right]$ and set
\[
Z = Z_{\varepsilon} = \{z = z_{\rho, \varepsilon}(r)=\zeta_{\varepsilon}(r)U_{\rho,\e}(r-\rho): \rho \in \Omega_\e\}.
\]
\smallskip

The purpose is to investigate concentration solutions to \(\eqref{disanjiediyige}\)--\(\eqref{disanjiedierge}\) taking the form
\[
u = z + \omega, \quad z = z_{\rho, \varepsilon} \in Z, \quad \omega \perp T_{z} Z.
\]

\medskip

We first consider \eqref{disanjiediyige} without constraint.

\smallskip

Given any \(\rho \in \Omega_\e\), we consider the equation \(J_{\varepsilon}'(z + \omega) \in T_{z} Z\), i.e., 
\begin{equation}\label{alphadefangcheng}
J_{\varepsilon}'(z + \omega) = \alpha \dot{z}, \quad \dot{z} = \frac{\partial z}{\partial \rho},   
\end{equation}
for some \(\alpha \in \mathbb{R}\). We use the orthogonal projection $P=P_{\rho,\e}$ onto \(W:=(T_{z} Z)^{\perp}\). It is expedient to write \eqref{alphadefangcheng} in the form 
\begin{equation*}
P J_{\varepsilon}'(z + \omega) = 0,
\end{equation*}
which is equivalent to
$$PJ_{\varepsilon}'(z)+PH_w+PJ_{\varepsilon}''(z)[\omega]=0,$$
where 
$$H_w=J_{\varepsilon}'(z+\omega)-J_{\varepsilon}'(z)-J_{\varepsilon}''(z)[\omega].$$
We shall demonstrate (see Proposition \ref{prop4.2}) that for sufficiently small $\e$, \(P J_{\varepsilon}''(z)\) is invertible and consequently, setting 
$$\mathcal{A}_{\varepsilon}(\omega) := -\left[P J_{\varepsilon}''(z)\right]^{-1} (PJ_{\varepsilon}'(z)+PH_w).$$
Solving \eqref{alphadefangcheng} amounts to finding the solutions of \(\omega = \mathcal{A}_{\varepsilon}(\omega)\).

\smallskip

Recall that \(\eqref{shuaijianxing}\) leads to
\be\label{zhishushuaijianxingzhi}
z(r) \leq C e^{-\lambda_{0} |r - \rho|},
\ee
where $\lambda_0$ is a positive constant such that for $\e$ small enough, $\inf\{1+\e^2V(\e r) : r \in \mathbb{R}^{+}\}\geq\lambda_{0}^{2}> 0$.
Selecting \(\eta > 0\), we assume that
\[
\lambda_{1} := \lambda_{0} - \eta > \frac{\lambda_{0}}{\min\{p, 2\}}.
\]
\smallskip

It will be shown that there exists a positive constant $\gamma>0$ such that, by setting 
\begin{equation}\label{C_edingyi}
E_{\varepsilon} := \left\{\omega \in H_{r}^{1} : \|\omega\|_{H_{r}^{1}} \leq \gamma \varepsilon^3 \|z\|_{H_{r}^{1}}, \,\,|\omega(r)| \leq \gamma e^{-\lambda_{1}(\rho - r)} \text{ for } r \in [0, \rho]\right\},
\end{equation}
for all \(\varepsilon\) small and all \(\rho \in \Omega_\e\), the map \(\mathcal{A}_{\varepsilon}\) is a contraction on \(E_{\varepsilon}\). Hence, it possesses a unique fixed point \(\omega_{\rho, \varepsilon}\in E_{\varepsilon}\) satisfying $$J_{\varepsilon}'(z_{\rho,\e}(r) + \omega_{\rho,\e}(r)) = \alpha \dot{z}.$$
\medskip

Furthermore, it will be demonstrated that $\omega$ is of class $C^1$ with respect to $\rho$, and it follows that
\begin{equation}\label{C_1}
\left\| \frac{\partial \omega}{\partial \rho} \right\| = o\left(\|\dot{z}\|\right).    
\end{equation}
\medskip

To see that $\alpha=0$, we define the finite-dimensional functional
\[
\Psi_{\varepsilon}(\rho) := J_{\varepsilon}\left(z_{\rho, \varepsilon} + \omega_{\rho, \varepsilon}\right).
\]
and show the following Proposition.

\begin{Prop} \label{prop2.2}
For sufficiently small \(\varepsilon\), if \(\rho_{\varepsilon}\) is a stationary point of \(\Psi_{\varepsilon}\), then \(\tilde{u}_{\varepsilon}(r) = z_{\rho_{\varepsilon}, \varepsilon}(r) + \omega_{\rho_{\varepsilon}, \varepsilon}(r)\) is a critical point of \(J_{\varepsilon}\), which is a solution of equation \eqref{disanjiediyige}.
\end{Prop}

\begin{proof}
As for $\rho_\e$ mentioned above, we infer that
\[
0 = \frac{\partial \Psi}{\partial \rho} = J'(z_{\rho,\e} + \omega_{\rho,\e})\left[\frac{\partial z_{\rho}}{\partial \rho} + \frac{\partial \omega_{\rho}}{\partial \rho}\right] = \alpha \left\| \frac{\partial z_{\rho}}{\partial \rho} \right\|^{2} + \alpha \left( \frac{\partial z_{\rho}}{\partial \rho}, \frac{\partial \omega_{\rho}}{\partial \rho} \right),
\]
where \(\alpha = \alpha(\rho_{\varepsilon})\) is obtained from \eqref{alphadefangcheng}. Combining with \eqref{C_1}, we infer that for $\e$ small enough, \(\alpha = 0\) and the Proposition \ref{prop2.2} follows.
\end{proof}
\medskip

\begin{Rem}
Proposition \ref{prop2.2}, together with the asymptotic expansion of $\Psi_\varepsilon$ given in Lemma \ref{lem5.1},
\[
\varepsilon^{3n-3} \Psi_{\varepsilon}(\rho)
= C_{0}\varepsilon^{2} M_\varepsilon(\varepsilon \rho) + o(1),
\qquad \rho \in \Omega_\varepsilon,
\]
shows that $\Psi_\varepsilon$ has a stationary point $\rho_\varepsilon$ with
$\rho_\varepsilon \sim t_\varepsilon/\varepsilon$. This asymptotic behavior
provides the main motivation for introducing the configuration set
$\Omega_\varepsilon$.

\end{Rem}

\medskip

%Combining with Proposition \ref{prop2.2}, the stationary point $\rho_\e$ of \(\Psi_{\varepsilon}\) yields a radial solution $\tilde{u}_{\varepsilon}(r) = z_{\rho_{\varepsilon}, \varepsilon}+ \omega_{\rho_{\varepsilon}, \varepsilon}$ to \eqref{disanjiediyige}. By rescaling argument, we infer that  $u_{\varepsilon}(r) = \tilde{u}_{\varepsilon}(r/\varepsilon)$ is also a radial solution to \eqref{fangcheng1} concentrating near the sphere $|x| = t_\varepsilon$.

%As for the constraint condition \eqref{disanjiedierge}, we next consider \(a > 0\) appropriately large. More explicitly, for \(a \geq \left( \frac{2^n C_2^{n-1} B}{\e_0^{\frac{4}{p-1} - 3 + 2n}} \right)^{\frac{p-1}{2}}\) and define\begin{equation*}F(\varepsilon):= BC_\varepsilon^{n-1}\varepsilon^{3-3n}-a^{\frac{2}{p-1}}\varepsilon^{\frac{4}{p-1}-n}+o\left(\varepsilon^{3-3n}\right).\end{equation*}By direct computation, we show that the constraint condition \eqref{disanjiedierge} is equivalent to \(F(\varepsilon)=0\).
%Therefore, we will proof the existence of a sequence $\{\e_a\}$ such that $F(\e_a)$=0, which means that $\tilde{u}_{\e_a}(r)$ satisfies \eqref{disanjiediyige}-\eqref{disanjiedierge}. That is to say, $u_{\e_a}(r)=\tilde{u}_{\e_a}(\frac{r}{\e})$ constitutes a radial solution to \eqref{fangcheng1}-\eqref{fangcheng11} concentrating near the sphere $|x|=t_\e$.

\subsection{Essential estimates}
This section is devoted to deriving several key estimates that are indispensable for carrying out the reduction.
\begin{Lem}
For any \(\varepsilon > 0\) small enough, \(\rho \in \Omega_\e\), \(\omega \in E_{\varepsilon}\) and \(r > 0\), the following hold:
\begin{equation}\label{zrho}
\|z_{\rho, \varepsilon}\|_{H_{r}^{1}} \sim \varepsilon^{(3-3n)/2}, 
\end{equation}
\begin{equation}\label{wfanshu}
\|\omega\|_{H_{r}^{1}} \lesssim \varepsilon^3 \|z_{\rho, \varepsilon}\|_{H_{r}^{1}} \sim \varepsilon^{(9-3n)/2},     
\end{equation}and
\begin{equation}\label{wzhudian}
|\omega(r)| \leq C \varepsilon^3, \,\text{for}\,\,
r \geq 0,    
\end{equation}
where \(C\) relies only on \(n\) and the constant \(\gamma\).
\end{Lem}

\begin{proof}
The definition of \(z_{\rho, \varepsilon}\) and the exponential decay of \(U_{\rho,\e}\) as \(|x| \to \infty\) show 
\[
\|z_{\rho, \varepsilon}\|_{H_{r}^{1}}^{2} = \int_{0}^{+\infty} r^{n-1}\left(|z'|^{2} +(1+\e^2 V(\varepsilon r)) z^{2}\right) dr \sim \rho^{n-1}.
\]
In view of \(\rho \in \Omega_\varepsilon\), \(\rho \sim \varepsilon^{-3}\) which will be used frequently holds, so \eqref{zrho} follows. Equation \eqref{wfanshu} is a direct corollary of \eqref{zrho}, coupled with the condition that \(\omega \in E_{\varepsilon}\). 

To obtain a uniform estimate of \(|\omega(r)|\), note that for \(\varepsilon\) small enough such that \(\frac{C_1}{8\e^3}> 1\). Employing \(\eqref{jingxiangjiedeguji}\), we obtain
\[
|\omega(r)| \leq c r^{(1-n)/2} \|\omega\|_{H_{r}^{1}} \leq c' \varepsilon^{(3n-3)/2} \|\omega\|_{H_{r}^{1}}, \text{ for}\,\, r \geq \frac{C_1}{8\varepsilon^3}.
\]
In light of \eqref{wfanshu}, this directly yields that
\[
|\omega(r)| \leq C \varepsilon^3, \text{ for  } r \geq \frac{C_1}{8\varepsilon^3}.
\]

Additionally, by the definition of \(E_{\varepsilon}\) and as \(\rho \in \Omega_\e\) which produces \(\rho \geq\frac{C_1}{2\varepsilon^3}\), we have 
\[
|\omega(r)| \leq \gamma e^{-\lambda_{1}(\rho - r)} \leq \gamma e^{-\lambda_1\frac{3C_1}{8\e^3}}\leq C \varepsilon^3, \text{ for  } r \leq \frac{C_1}{8\varepsilon^3}.
\]
Consequently, \eqref{wzhudian} holds.
\end{proof}
\begin{Rem}
The pointwise estimate of $\omega(r)$ in \eqref{wzhudian} allows us to handle any exponent $p>1$, including critical or supercritical cases (see subsection \ref{5.1}).
\end{Rem}

\smallskip

\begin{Prop}

For $\e$ small enough, \(\rho \in \Omega_\e\) and $\omega\in E_{\varepsilon}$, the following estimates are satisfied:
\begin{itemize}
\item[(A1)] \(\|J_{\varepsilon}'(z_{\rho, \varepsilon})\| \lesssim \varepsilon^3\|z_{\rho, \varepsilon}\|_{H_{r}^{1}} \sim \varepsilon^{(9-3n)/2}\);\vspace{0.3em}
\item[(A2)] \(\|J_{\varepsilon}''(z_{\rho, \varepsilon} + s \omega)\| \leq C\), \((0 \leq s \leq 1)\);\vspace{0.3em}
\item[(A3)] \(\|J_{\varepsilon}'(z_{\rho, \varepsilon} + \omega)\| \lesssim \varepsilon^{(9-3n)/2}\);\vspace{0.3em}
\item[(A4)] \(\|J_{\varepsilon}''(z_{\rho, \varepsilon} + s \omega) - J_{\varepsilon}''(z_{\rho, \varepsilon})\| \lesssim \varepsilon^{3\wedge 3(p-1)}\), \((0 \leq s \leq 1)\);\vspace{0.3em}
\item[(A5)] \(\|J_{\varepsilon}'(z_{\rho, \varepsilon} + \omega) - J_{\varepsilon}'(z_{\rho, \varepsilon}) - J_{\varepsilon}''(z_{\rho, \varepsilon})[\omega]\| \lesssim \varepsilon^{3 \wedge 3(p-1)} \|\omega\|_{H_{r}^{1}}\).
\end{itemize}
\end{Prop}
~\\
Throughout the paper, we often omit the subscripts $\rho$ and $\e$ for the sake of simplicity.
\begin{proof}
\textbf{Proof of (A1).} Given any \(v \in H_{r}^{1}\), we get
\[
\begin{aligned} 
& J_{\varepsilon}'(z)[v] \\ 
= &\int_{0}^{+\infty} r^{n-1}\left(z' v' + (1+\e^2V(\varepsilon r)) z v - z^{p} v\right) dr \\ 
= &-\int_{0}^{+\infty} v \left(r^{n-1} z'\right)' dr + \int_{0}^{+\infty} r^{n-1}\left((1+\e^2V(\varepsilon r)) z v - z^{p} v\right) dr \\ 
= &-(n-1) \underbrace{\int_{0}^{+\infty} r^{n-2} z' v dr}_{I_1(v)}\underbrace{-\int_{0}^{+\infty} r^{n-1} z'' v dr + \int_{0}^{+\infty} r^{n-1}\left((1+\e^2V(\varepsilon r)) z v - z^{p} v\right) dr}_{I_2(v)}.
\end{aligned}
\]

Applying the H{\"o}lder inequality to the first term (denoted $I_1(v)$) gives,
\[
|I_1(v)| \leq C \|v\|_{H_{r}^{1}} \left( \int_{0}^{+\infty} \left(r^{(n-3)/2} z'\right)^{2} dr \right)^{1/2}.
\]
Given that $z$ exhibits exponential decay away from \(r = \rho\) and \(\rho \in \Omega_\varepsilon\), we deduce that
\[
\int_{0}^{+\infty} \left(r^{(n-3)/2} z'\right)^{2} dr = \int_{0}^{+\infty} r^{-2} \cdot r^{n-1} |z'|^{2} dr \sim \rho^{-2} \|z\|_{H_{r}^{1}}^{2} \sim \varepsilon^{6} \|z\|_{H_{r}^{1}}^{2}.
\]
Thus, employing \eqref {zrho}, we derive
\begin{equation}
\label{A_0}
\sup\left\{|I_1(v)| : \|v\|_{H_{r}^{1}} \leq 1\right\} \lesssim \varepsilon^3 \|z\| \sim \varepsilon^{(9-3n)/2}.
\end{equation}

For the second term (denoted $I_2(v)$), note \(z = \zeta \cdot U(r - \rho)\), so we have
\begin{align*} 
I_2(v) = & \underbrace{-\int_{0}^{+\infty} r^{n-1}\left(\zeta'' U + 2 \zeta' U'\right) v dr}_{I_3(v)} \\ 
& + \underbrace{\int_{0}^{+\infty} r^{n-1}\left((1+\varepsilon^2V(\varepsilon r))\zeta U v - (\zeta U)^{p} v\right) dr - \int_{0}^{+\infty} r^{n-1} \zeta U'' v dr}_{I_{4}(v)},
\end{align*}
where \(U\) denotes \(U_{\rho,\e}(r - \rho)\). As the support of \(\zeta'\) lies in the interval $\left[\frac{C_1}{16\e^3},\frac{C_1}{8\e^3}\right]$ and \(U\) decays exponentially to zero as \(r \to \infty\), we deduce that 
\begin{equation}\label{A_2}
\sup\left\{|I_3(v)| : \|v\|_{H_{r}^{1}} \leq 1\right\} \lesssim e^{-c/\varepsilon^3}\lesssim\e^3\|z\|.     
\end{equation}

Lastly, by \eqref{U_rhovar} we infer 
\[
I_4(v) = \e^2\int_{0}^{+\infty} r^{n-1} (V(\varepsilon r) - V(\varepsilon \rho)) \zeta U v dr =\e^2 \int_{0}^{+\infty} r^{n-1} (V(\varepsilon r) - V(\varepsilon \rho)) z v dr.
\]
In view of \(V'\) being bounded,
\[
|I_4(v)| \leq C \e^3\|v\|_{H_{r}^{1}} \left( \int_{0}^{+\infty} r^{n-1} z^{2} dr \right)^{1/2}.
\]
It is further straightforward to deduce that
\begin{equation}\label{A_3}
\sup\left\{|I_4(v)| : \|v\|_{H_{r}^{1}} \leq 1\right\} \lesssim \varepsilon^3\|z\| . \end{equation}
Combining \eqref{A_0}, \eqref{A_2}, and \eqref{A_3}, we readily establish the assertion (A1).

\textbf{Proof of (A2).} For any \(v \in H_{r}^{1}\), we first estimate 
\[
\begin{aligned} 
|J_{\varepsilon}''(z + s \omega)[v, v]| & = \left| \int_{0}^{+\infty} r^{n-1}\left(|v'|^{2} + (1+\e^2V(\varepsilon r)) v^{2} - p |z + s \omega|^{p-1} v^{2}\right) dr \right| \\ 
& \leq \|v\|_{H_{r}^{1}}^{2} + p \left| \int_{0}^{+\infty} r^{n-1} |z + s \omega|^{p-1} v^{2} dr \right|.
\end{aligned}
\]

Seeing that \eqref{wzhudian} entails \(|z(r) + s \omega(r)| \leq C\), this in turn implies that 
\[
\left| \int_{0}^{+\infty} r^{n-1} |z + s \omega|^{p-1} v^{2} dr \right| \leq C \|v\|_{H_{r}^{1}}^{2}.
\]
We obtain \(|J_{\varepsilon}''(z + s \omega)[v, v]| \leq C \|v\|_{H_{r}^{1}}^{2}\), whence we conclude the proof of (A2).

\textbf{Proof of (A3).} Using the mean value theorem 
\[
J_{\varepsilon}'(z + \omega) = J_{\varepsilon}'(z) + \int_{0}^{1} J_{\varepsilon}''(z + s \omega)[\omega] ds,
\]
we compute 
\[
\|J_{\varepsilon}'(z + \omega)\| \leq \|J_{\varepsilon}'(z)\| + \sup_{0 \leq s \leq 1} \|J_{\varepsilon}''(z + s \omega)\| \cdot \|\omega\|.
\]
Invoking (A1), (A2) and \eqref{wfanshu}, we obtain the result (A3).

\textbf{Proof of (A4).} We additionally obtain
\[
\begin{aligned} 
|J_{\varepsilon}''(z + s \omega)[v, v] - J_{\varepsilon}''(z)[v, v]| & = \left| \int_{0}^{+\infty} r^{n-1}\left(p |z + s \omega|^{p-1} - p z^{p-1}\right) v^{2} dr \right| \\ 
& \leq p \int_{0}^{+\infty} r^{n-1} \left( (|z| + |\omega|)^{p-1} - z^{p-1} \right) v^{2} dr \\ 
& \leq C \int_{0}^{+\infty} \left( |\omega| + |\omega|^{p-1} \right) r^{n-1} v^{2} dr.
\end{aligned}
\]
By applying \eqref {wzhudian}, we derive
\[
|J_{\varepsilon}''(z + s \omega)[v, v] - J_{\varepsilon}''(z)[v, v]| \leq C \left( \varepsilon^3 + \varepsilon^{3(p-1)} \right) \|v\|_{H_{r}^{1}}^{2},
\]
giving the desired result (A4).

\textbf{Proof of (A5).} Similarly, using the mean value theorem for the difference of first derivatives:
\[
J_{\varepsilon}'(z + \omega) - J_{\varepsilon}'(z) = \int_{0}^{1} J_{\varepsilon}''(z + s \omega)[\omega] ds.
\]
As a consequence, 
\[
J_{\varepsilon}'(z + \omega) - J_{\varepsilon}'(z) - J_{\varepsilon}''(z)[\omega] = \int_{0}^{1} \left( J_{\varepsilon}''(z + s \omega) - J_{\varepsilon}''(z) \right)[\omega] ds.
\]
Furthermore, by employing (A4), we can easily verify (A5).
\end{proof}

\smallskip

\begin{Rem}\label{rem3.4}

As for the estimate (A1), when $n=2$ we obtain
\[
\|J_{\varepsilon}'(z_{\rho,\varepsilon})\|
\,\lesssim\, \varepsilon^{3}\,\|z_{\rho,\varepsilon}\|_{H_{r}^{1}}
\,\sim\, \varepsilon^{3/2},
\]
which is a small quantity as $\varepsilon\to0$. In this case one can follow
the contraction mapping scheme of \cite{gtzcv}. When $n\ge3$, however, the
same estimate yields a bound of order $\varepsilon^{(9-3n)/2}$, which is no
longer small as $\varepsilon\to0$ (it is of order $1$ for $n=3$ and blows up
for $n>3$). Hence one cannot directly use (A1) to define the fixed point set
and carry out the finite-dimensional reduction as in \cite{gtzcv}. To
overcome this difficulty, we refine the approximate solution and introduce a
new fixed point set as in \eqref{C_edingyi}, for which we derive uniformly
bounded estimates for $\omega$, see \eqref{wzhudian}. This allows us to
recover a contraction mapping argument also in higher dimensions.

\end{Rem}
\medskip

\subsection{Fixed point of \(\mathcal{A}_{\varepsilon}\) }
This section focuses on the contraction mapping argument on \(E_{\varepsilon}\). Our goal is to find a function \(\omega=\omega_{\rho,\e}\) satisfying the following two conditions 
\begin{itemize}
\item[i)] \(PJ_{\varepsilon}'(z + \omega) =0\),
\item[ii)] \(\omega \in W:=(T_{z} Z)^{\perp}\).
\end{itemize}
\medskip

%\subsection{Initial estimates concerning \(\dot{z}\).}
We start by obtaining an estimate on the term $\dot{z}$. Since \(z_{\rho, \varepsilon}(r) = \zeta_{\varepsilon}(r) U_{1+\e^2V(\varepsilon \rho)}(r - \rho)\), it holds that
\[
\frac{\partial}{\partial \rho} z_{\rho, \varepsilon} = \zeta_{\varepsilon}(r) \left( \varepsilon^3 V'(\varepsilon \rho) \left. \left( \frac{\partial}{\partial V} U_{V} \right) \right|_{1+\e^2V(\varepsilon \rho)} (r - \rho) - \frac{\partial}{\partial r} U_{1+\e^2V(\varepsilon \rho)}(r - \rho) \right).
\]
By applying the fact that exponential decay of \(U_{\rho,\e}\) and $\rho\sim\e^{-3}$, we get
\begin{equation}\label{dot{z}}
\left\| \frac{\partial}{\partial \rho} z_{\rho, \varepsilon} \right\|^{2} \sim \rho^{n-1} \int_{0}^{+\infty} \left( U_{1+\e^2V(\varepsilon \rho)}''^{2} + (1+\e^2V(\e r))U_{1+\e^2V(\varepsilon \rho)}'^{2} \right) \sim \frac{1}{\varepsilon^{3(n-1)}}. 
\end{equation}
Note that the function $\dot{z}$ and its second-order derivatives are uniformly bounded and exhibit exponential decay away from $\rho$.

\medskip

Firstly, we prove that $J_{\varepsilon}''(z)$ is coercive on the orthogonal complement in $H_r^1$ of $\{t z\} \oplus T_z Z$.

\smallskip
\begin{Lem} \label{lem4.1}
There exists a positive constant \(\delta\) satisfying that, for every \(\rho \in \Omega_\e\) and for \(\varepsilon\) sufficiently small, the following holds
\[
J_{\varepsilon}''(z)[u, u] \geq \delta \|u\|^{2}, \text{ for all } u \perp \{t z\} \oplus T_{z} Z.
\]
\end{Lem}

\begin{proof}
For some \(\mu > 0\) to be determined in what follows. We introduce a cutoff function \(\xi_{\rho, \mu}: \mathbb{R}^{+} \to \mathbb{R}\) which fulfills
\begin{equation*}
\left\{
\begin{array}{ll} 
\xi_{\rho, \mu}(r) = 0, & \text{for } r \in [0, \rho - 2\mu] \cup [\rho + 2\mu, \infty], \\ 
\xi_{\rho, \mu}(r) = 1, & \text{for } r \in [\rho - \mu, \rho + \mu], \\ 
|\xi_{\rho, \mu}'(r)| \leq \frac{2}{\mu}, & \text{for } r \in [\rho - 2\mu, \rho - \mu] \cup [\rho + \mu, \rho + 2\mu], \\ 
|\xi_{\rho, \mu}''(r)| \leq \frac{4}{\mu}, & \text{for } r \in [\rho - 2\mu, \rho - \mu] \cup [\rho + \mu, \rho + 2\mu].
\end{array}
\right. 
\end{equation*}

In what follows, and whenever no confusion can arise, we shall simply denote
$\xi_{\rho,\mu}$ by $\xi$. For any $u \in H^{1}$, we can rewrite it as
\begin{equation*}
\|u\|^{2} = \|\xi u\|^{2} + \|(1 - \xi) u\|^{2} + 2(\xi u, (1 - \xi) u),
\end{equation*}
and 
\begin{equation}
\label{fanhandefenjie}
J_{\varepsilon}''(z)[u, u] = J_{\varepsilon}''(z)[\xi u, \xi u] + J_{\varepsilon}''(z)[(1 - \xi) u, (1 - \xi) u] + 2 J_{\varepsilon}''(z)[\xi u, (1 - \xi) u].
\end{equation}

On the other side, if $u$ is orthogonal to both $z$ and $\dot{z}$, and given that $z$ and $\dot{z}$ exhibit exponential decay (see \eqref{zhishushuaijianxingzhi} and \eqref{dot{z}}), it follows readily from the H\"older inequality that

\begin{equation}
\label{uz}
|(\xi u, z)| = |-((1 - \xi) u, z)| \leq C e^{-\mu} \|z\| \|u\| \leq C e^{-\mu} \varepsilon^{\frac{3-3n}{2}} \|u\|,
\end{equation}
and
\begin{equation}
\label{udotz}
|(\xi u, \dot{z})| = |-((1 - \xi) u, \dot{z})| \leq C e^{-\mu} \|\dot{z}\| \|u\| \leq C e^{-\mu} \varepsilon^{\frac{3-3n}{2}} \|u\|,    
\end{equation}
where \(C\) is a positive constant that is independent of \(\varepsilon\), $u$, \(\mu\), as well as \(\rho \in\Omega_{\varepsilon}\). 

Additionally, since \(V\in C^1(\R^+,\R)\), we have the following relations
\begin{equation}
\label{Vhern-1}
|V(\varepsilon r) - V(\varepsilon \rho)| \leq C \varepsilon \mu, \quad |r^{n-1} - \rho^{n-1}| \leq C \mu \varepsilon^{6-3n}, \text{ for }\, r \in [\rho - 2\mu, \rho + 2\mu].
\end{equation}
Formulas \eqref{uz}, \eqref{udotz} and the second equation in \eqref{Vhern-1} collectively entail that
\begin{equation}
\label{diyige}
\rho^{n-1} \int \left( ((\xi u)')^{2}(r) + (\xi u)^{2}(r) \right) dr = \|\xi u\|^{2} + O\left( \frac{1}{\mu} + \varepsilon^3 \mu \right) \|u\|^{2},
\end{equation}
\begin{equation}
\label{dierge}
\rho^{n-1} \int \left( (\xi u)'(r) U_{1+\e^2V(\varepsilon \rho)}'(r) + (\xi u)(r) U_{1+\e^2V(\varepsilon \rho)}(r) \right) dr = O\left( \frac{1}{\mu} + \varepsilon^3 \mu \right) \|u\|,
\end{equation}
and
\begin{equation}
\label{disange}
\rho^{n-1} \int \left( (\xi u)'(r) U_{1+\e^2V(\varepsilon \rho)}''(r) + (\xi u)(r) U_{1+\e^2V(\varepsilon \rho)}'(r) \right) dr = O\left( \frac{1}{\mu} + \varepsilon^3 \mu \right) \|u\|.
\end{equation}

For arbitrarily chosen \(u(r) \in H^{1}(\mathbb{R})\), we observe that 
\begin{align*}
\begin{split}
J_{\varepsilon}''(z)[u, u]-\rho^{n-1} \bar{J}''(z)[u,u] & = \int_{0}^{\infty} r^{n-1}\left( (u')^{2} +(1+\e^2 V(\varepsilon r)) u^{2} - p z^{p-1} u^{2} \right) dr \\ 
& \quad - \rho^{n-1} \int_{0}^{\infty} \left( (u')^{2} + (1+\e^2V(\varepsilon \rho)) u^{2} - p z^{p-1} u^{2} \right) dr, 
\end{split}\end{align*}
where $\bar{J}(u)$ is the functional corresponding to \eqref{U_rhovar}. Therefore, because of \eqref{Vhern-1} and the decay of \(z\), we see that 
\begin{equation}
\label{fenjiedexiayibu}
\begin{aligned}
J_{\varepsilon}''(z)[\xi u, \xi u] & = \rho^{n-1} \int_{0}^{\infty} \left( ((\xi u)')^{2} + (1+\e^2V(\varepsilon \rho)) (\xi u)^{2} - p z^{p-1} (\xi u)^{2} \right) dr \\ 
& \quad + O\left( \frac{1}{\mu} + \varepsilon^3 \mu \right) \|u\|^{2}.
\end{aligned}
\end{equation} 

Combining with \eqref{diyige}-\eqref{disange}, \eqref{fenjiedexiayibu} and Proposition \ref{prop2.111}, there exists $\delta>0$, such that 
\[
J_{\varepsilon}''(z)[\xi u, \xi u] \geq \delta \|\xi u\|^{2} + O\left( \frac{1}{\mu} + \varepsilon^3 \mu \right) \|u\|^{2}.
\]
Considering the exponential decay of \(z\) again, we also infer that 
\[
J_{\varepsilon}''(z)[(1 - \xi) u, (1 - \xi) u] = \|(1 - \xi) u\|^{2} + O\left( \frac{1}{\mu} + \varepsilon^3 \mu \right) \|u\|^{2},
\]
and
\[
J_{\varepsilon}''(z)[\xi u, (1 - \xi) u] = (\xi u, (1 - \xi) u) + O\left( \frac{1}{\mu} + \varepsilon^3 \mu \right) \|u\|^{2}.
\]

In light of \eqref{fanhandefenjie}, we conclude that 
\[
\begin{aligned} 
J_{\varepsilon}''(z)[u, u] & \geq \delta \|\xi u\|^{2} + \|(1 - \xi) u\|^{2} + 2(\xi u, (1 - \xi) u) + O\left( \frac{1}{\mu} + \varepsilon^3 \mu \right) \|u\|^{2}\\ 
& \geq \delta \|u\|^{2} + O\left( \frac{1}{\mu} + \varepsilon^3 \mu \right) \|u\|^{2}.
\end{aligned}
\]
Consequently, we can choose \(\mu\) sufficiently large, then $J_{\varepsilon}^{\prime\prime}(z)$ is positive definite in \(u\), with \(\varepsilon\) is sufficiently small. We see that Lemma \ref{lem4.1} follows.
\end{proof}

It is simple to check that
\[
J_{\varepsilon}''(z)[z, z] = (1 - p) \int z^{p+1} + o_{\varepsilon}(1).
\]
\smallskip

Combining the above equation with Lemma~\ref{lem4.1}, we obtain
Proposition~\ref{prop4.2}.

\smallskip
\begin{Prop} \label{prop4.2}
Let \(\varepsilon\) be sufficiently small and for any \(\rho \in \Omega_\e\), the operator \(P J_{\varepsilon}''(z)\) is invertible on \(W\) with uniformly bounded inverse. That is to say, it holds that
\[
\|L_{\varepsilon}\| \leq C, \text{ where } L_{\varepsilon} := -\left(P J_{\varepsilon}''(z)\right)^{-1},
\]
where \(C\) is a positive constant independent of $\e$.
\end{Prop}
%\subsection{Existence of $\omega$}
\medskip

We now proceed to prove  $\omega=\mathcal{A}_{\varepsilon}(\omega)$ is solvable on \(E_{\varepsilon}\) provided \(\varepsilon\) is small enough. 
\begin{proposition} \label{prop4.3}
There exists a positive constant \(\gamma\), such that for \(\varepsilon\) sufficiently small and  \(\rho \in \Omega_{\varepsilon}\), there exists a function \(\omega = \omega(z_{\rho, \varepsilon})\) satisfying
\begin{itemize}
\item[i)] $P J_{\varepsilon}'(z + \omega) = 0$,
\item[ii)] $\omega\in W$,
\item[iii)] \(\omega \in E_{\varepsilon} = \left\{\omega \in H_{r}^{1} : \|\omega\| \leq \gamma \varepsilon^3 \|z\|, |\omega(r)| \leq \gamma e^{-\lambda_{1}(\rho- r)} \text{ for } r \in [0, \rho]\right\}\),
\item[iv)] \(\|\omega\| \leq C \|J_{\varepsilon}'(z_{\rho})\|\).
\end{itemize}
Besides, \(\omega\) is of class \(C^{1}\) with respect to \(\rho\), and \(\|\partial \omega / \partial \rho\| = o(\|\dot{z}\|)\) as $\e\to0$.
\end{proposition}

\begin{proof}

Step 1. Contraction mapping.

Applying (A1), (A5) and Proposition \ref{prop4.2}, we can write
\begin{align}
\label{sepwdefanshu}
\begin{split}
\|\mathcal{A}_{\varepsilon}(\omega)\| &\leq C(\|J_{\varepsilon}'(z_{\rho})\|+\|H_w\|) \\&\leq C\left( \e^3\|z\| + \varepsilon^{3 \wedge 3(p-1)} \|\omega\| \right) \\ 
&\leq C\left( \varepsilon^3 \|z\| + \varepsilon^{3 \wedge 3(p-1)} \gamma\e^3\|z\| \right)\\
&=C_0\e^3\|z\|(1+\gamma\varepsilon^{3 \wedge 3(p-1)}), \,\text{for}\, \omega \in E_{\varepsilon}. 
\end{split}
\end{align}
So it is true that for $\gamma$ large enough and $\e$ small enough,
$$\|\mathcal{A}_{\varepsilon}(\omega)\|\leq\gamma\e^3\|z\|.$$

Next,  for any \(\omega_1\), \(\omega_2\in E_{\varepsilon}\), we show
\[
\mathcal{A}_{\varepsilon}(\omega_1) - \mathcal{A}_{\varepsilon}(\omega_2) = L_{\varepsilon}\left( P J_{\varepsilon}'(z + \omega_1) - P J_{\varepsilon}''(z)[\omega_1] - P J_{\varepsilon}'(z + \omega_2) + P J_{\varepsilon}''(z)[\omega_2] \right).
\]
By making use of
\[
\begin{aligned} 
& J_{\varepsilon}'(z + \omega_1) - J_{\varepsilon}''(z)[\omega_1] - J_{\varepsilon}'(z +\omega_2) + J_{\varepsilon}''(z)[\omega_2] \\ 
=& \int_{0}^{1} \left( J_{\varepsilon}''(z +\omega_2 + s(\omega_1 - \omega_2)) - J_{\varepsilon}''(z) \right)[\omega_1 - \omega_2] ds,
\end{aligned}
\]
and following the same line of reasoning as in the proof of (A4),
\[
\left\|J_{\varepsilon}'(z + \omega_1) - J_{\varepsilon}''(z)[\omega_1] - J_{\varepsilon}'(z +\omega_2) + J_{\varepsilon}''(z)[\omega_2] \right\| \leq C \varepsilon^{3 \wedge 3(p-1)} \|\omega_1 - \omega_2\|.
\]
From the above arguments, we conclude that

\[
\|\mathcal{A}_{\varepsilon}(\omega_1) - \mathcal{A}_{\varepsilon}(\omega_2)\| \leq C \varepsilon^{3 \wedge 3(p-1)} \|\omega_1 - \omega_2\|.
\]

With the norm estimates at hand, it remains to prove that

\[
\omega \in E_{\varepsilon} \Rightarrow |(\mathcal{A}_{\varepsilon} \omega)(r)| \leq \gamma e^{-\lambda_{1}(\rho - r)} \text{ for } 0 \leq r \leq \rho.
\]
What is more, in this part, we will provide a more precise range for $\gamma$.

\medskip

Step 2. Pointwise estimate.

Employing the definition of \(L_{\varepsilon}\), it is immediately apparent that for any \(f \in H_{r}^{1}\), the following holds
\begin{equation}
\label{jugelizi}
\omega = L_{\varepsilon}(P f) \Leftrightarrow -P J_{\varepsilon}''(z)[\omega] = f - \frac{(f, \dot{z}) \dot{z}}{\|\dot{z}\|^{2}}. 
\end{equation}
Testing the equation \(P J_{\varepsilon}''(z)[\omega] = -f + \|\dot{z}\|^{-2}(f, \dot{z}) \dot{z}\) against any smooth function, we can see that \(\omega\) satisfies the following differential equation,
\begin{equation}\label{lizifangcheng}
-\Delta \omega + (1+\e^2V(\varepsilon r)) \omega - p z^{p-1} \omega = -\Delta(\alpha \dot{z}-f) +(1+\e^2 V(\varepsilon r))( \alpha \dot{z}-f) \text{ in } \mathbb{R}^{n}, 
\end{equation}
where
\[
\alpha = \|\dot{z}\|^{-2} \left( J_{\varepsilon}''(z)[\omega] + f, \dot{z} \right).
\]

As a result, if we define \(H_{z} := J_{\varepsilon}'(z)\) and for any $\omega\in E_{\varepsilon}$,
\[
\tilde{\omega} := \mathcal{A}_{\varepsilon}(\omega) =L_\e P(H_z+H_w).
\]
Combining with \eqref{jugelizi} and \eqref{lizifangcheng}, we shall show that \(\tilde{\omega} \in W\) fulfills the differential equation 
\begin{equation*}
\begin{aligned} 
& -\Delta \tilde{\omega} +(1+\e^2 V(\varepsilon r)) \tilde{\omega} - p z^{p-1} \tilde{\omega} \\ 
=& -\Delta\left(\beta \dot{z}-H_z-H_w\right) +(1+\e^2 V(\varepsilon r))\left(\beta \dot{z}-H_z-H_w\right) \text{ in } \mathbb{R}^{n},
\end{aligned}
\end{equation*}
with 
\be\label{betadingyi}
\beta = \|\dot{z}\|^{-2} \left( J_{\varepsilon}''(z)[\tilde{\omega}] + H_{z} + H_{\omega}, \dot{z} \right).
\ee

The functions \(H_{z}\) and \(H_{\omega}\) are defined by means of duality as 
\[
(H_{z}, v) = J_{\varepsilon}'(z)[v], \quad (H_{\omega}, v) = J_{\varepsilon}'(z + \omega)[v] - J_{\varepsilon}'(z)[v] - J_{\varepsilon}''(z)[\omega, v], \,\,\text{for}\,\, v \in H_{r}^{1}.
\]
Thus, they satisfy the two equations respectively 
\[
-\Delta H_{z} + (1+\e^2V(\varepsilon r)) H_{z} = -\Delta z +(1+\e^2 V(\varepsilon r)) z - z^{p}, \text{ in } \mathbb{R}^{n},
\]
and
\[
-\Delta H_{\omega} +(1+\e^2 V(\varepsilon r)) H_{\omega} = -\left( (z + \omega)^{p} - z^{p} - p z^{p-1} \omega \right), \text{ in } \mathbb{R}^{n}.
\]
In summary, \(\tilde{\omega}\) satisfies 
\begin{equation*}
\begin{aligned} 
-\Delta \tilde{\omega} +(1+\e^2 V(\varepsilon r) )\tilde{\omega} - p z^{p-1} \tilde{\omega} = &(z + \omega)^{p} - z^{p} - p z^{p-1} \omega+\beta(-\Delta \dot{z} + (1+\e^2V(\varepsilon r) )\dot{z}) \\ 
& -\left( -\Delta z + (1+\e^2V(\varepsilon r)) z - z^{p} \right), \text{ in } \mathbb{R}^{n}.
\end{aligned}
\end{equation*}

We now write, for some 
$N>0$ to be determined later,
\[
\tilde{\omega}(x) := \tilde{\omega}_{1}(x) + \tilde{\omega}_{2}(x), \text{ for } x \in B_{\rho - N}(0).
\]
Here, the functions $\tilde{\omega}_{1}(x)$ and $\tilde{\omega}_{2}(x)$ are defined respectively as the solutions of the two problems outlined below,
\begin{equation*}
\begin{cases}
-\Delta \tilde{\omega}_{1} +(1+\e^2 V(\varepsilon r)) \tilde{\omega}_{1} - p z^{p-1} \tilde{\omega}_{1} = 0, & \text{in } B_{\rho - N}(0), \\ 
\tilde{\omega}_{1} = \tilde{\omega}, & \text{on } \partial B_{\rho-N}(0),
\end{cases}
\end{equation*}
and
\begin{equation*}
\left\{
\begin{array}{ll} 
-\Delta \tilde{\omega}_{2} +(1+\e^2 V(\varepsilon r)) \tilde{\omega}_{2} - p z^{p-1} \tilde{\omega}_{2} = f_{z, \omega}, & \text{in } B_{\rho - N}(0), \\ 
\tilde{\omega}_{2} = 0, & \text{on } \partial B_{\rho - N}(0),
\end{array}
\right.
\end{equation*}
with 
\[
f_{z, \omega} = \left( (z + \omega)^{p} - z^{p} - p z^{p-1} \omega \right) +\beta(-\Delta \dot{z} + (1+\e^2V(\varepsilon r) )\dot{z}) - \left( -\Delta z +(1+\e^2 V(\varepsilon r)) z - z^{p} \right).
\]

We select positive numbers $\lambda_2$ and $\lambda_3$ satisfying the condition that
\begin{equation*}
\lambda_{1} < \lambda_{3} < \lambda_{2} < \lambda_{0} \text{ (recall that } \lambda_{1} < \lambda_{0} \text{)}. 
\end{equation*}

Let \(G(\cdot, y)\) be denoted as the Green's function of \(-\Delta + \lambda_{2}^{2}\) with pole \(y\). In other words, $G(x,y)$ is the solution of \(-\Delta u + \lambda_{2}^{2} u = \delta_{y}\). So it is true that 
\begin{equation}
\label{Green}
G(x, y) \leq C e^{-\lambda_{2}|x - y|} \frac{1}{|x - y|^{\frac{n-1}{2}}}, \text{ for } |x - y| \geq 1.
\end{equation}

Take $N$ such that
\begin{equation}
\label{xuanN}
\lambda_{0}^{2} - p \sup_{r \leq \rho - N} z^{p-1}(r) > \lambda_{2}^{2}\,\,\text{and}\,\,\rho-N>1.
\end{equation}
Arguing as in  \cite{cmp}, we compare \(\tilde{\omega}_{1}\) with the solution of
a constant–coefficient radial problem and use the maximum principle together
with the classical Bessel–type estimates 
  to obtain an exponential decay from the
boundary. More precisely, considering \(-\Delta u + V^{*}u=0\)  with
\[
V^{*}(r) := \bigl(1+\varepsilon^{2}V(\varepsilon r)\bigr) - p\,z^{p-1},
\quad R := \rho - N,
\]
we can obtain that, for all \(r\in[0,\rho-N]\),
\[
\tilde{\omega}_{1}(r)
\;\leq\;
C\,e^{-\lambda_{1}(R-r)}\,\tilde{\omega}_{1}(R)
\;=\;
C\,e^{\lambda_{1}N}\,e^{-\lambda_{1}(\rho-r)}\,\tilde{\omega}_{1}(\rho-N),
\]
where \(C>0\) and \(\lambda_{1}>0\) depend only on \(n\) and the spectral gap
associated with \(V^{*}\).

By \eqref{jingxiangjiedeguji}, it follows that
\begin{equation}
\label{tildew1guji}
\tilde{\omega}_{1}(r) \leq \tilde{C} \varepsilon ^3e^{-\lambda_{1}(\rho - r)}, \text{ for } r\in[0,\rho-N].
\end{equation}

Let \(\tilde{u}_{2}\) stand for the solution to the problem
\begin{equation*}
\left\{
\begin{array}{ll} 
-\Delta \tilde{u}_{2} + \lambda_{2}^{2} \tilde{u}_{2} = J_{1} + J_{2} + J_{3}, & \text{in } B_{\rho - N}(0), \\ 
\tilde{u}_{2} = 0, & \text{on } \partial B_{\rho - N}(0),
\end{array}
\right.
\end{equation*}
where 
\[
J_{1} := \left| (z + \omega)^{p} - z^{p} - p z^{p-1} \omega \right|,\]\[J_{2} := |\beta| |-\Delta \dot{z} + (1+\e^2V(\e r))\dot{z}|,
\]
and
\[
J_{3} := \left| -\Delta z +(1+\e^2 V(\varepsilon r)) z - z^{p} \right|.
\]
We infer from \eqref{xuanN} and the maximum principle that \(0 \leq |\tilde{\omega}_{2}| \leq \tilde{u}_{2}\). Now, we will estimate \(\tilde{u}_{2}\). For this purpose, we first examine the functions \(J_{1}\), \(J_{2}\) and \(J_{3}\). 

Moreover, the fact that $\omega\in E_{\varepsilon}$ implies
\be\label{f1}
\begin{aligned} 
J_{1} & = \left| (z + \omega)^{p} - z^{p} - p z^{p-1} \omega \right|\leq C |\omega|^{2 \wedge p} \\ 
& \leq C C_{1}^{2 \wedge p} e^{-(2 \wedge p) \lambda_{1}(\rho - r)}, \text{ for } r \leq \rho.
\end{aligned}
\ee
As in (A1), (A2), (A5), \eqref{dot{z}} and \eqref{betadingyi}, we can derive \(|\beta| \leq C\varepsilon^3\). As a consequence, 
\be\label{f2}
J_{2} = |\beta| |-\Delta \dot{z} + (1+\e^2V(\e r))\dot{z}| \leq C \varepsilon^3 e^{-\lambda_{2}(\rho - r)}, \text{ for } r \leq \rho.\ee
Furthermore, by the expression of \(z_{\rho, \varepsilon}(r)\), we can easily verify
\be\label{f3}
J_{3} = \left| -\Delta z +(1+\e^2 V(\varepsilon r)) z - z^{p} \right| \leq C e^{-\lambda_{2}(\rho - r)}, \text{ for } r \leq \rho. 
\ee
Henceforth, combining with \eqref{f1}-\eqref{f3}, we obtain 
\be\label{tildef}
J(r) := J_{1}+ J_{2}+ J_{3}\leq C e^{-\lambda_{2}(\rho - r)} + C C_{1}^{2 \wedge p} e^{-(2 \wedge p) \lambda_{1}(\rho - r)}, \text{ for } r \leq \rho.
\ee

Letting \(G_{B}(x,y)\) denotes the Green's function for the operator \(-\Delta u + \lambda_{2}^{2} u\) in \(B_{\rho - N}(0)\) with Dirichlet boundary conditions, by the maximum principle, it is easy to check that 
\be\label{tildeu2}
\tilde{u}_{2}(x) = \int_{B_{\rho - N}(0)} G_{B}(x, y) J(y) dy \leq \int_{\mathbb{R}^{n}} G(x, y) J(y) dy, \quad x \in B_{\rho - N}(0). \ee

Let \(\sigma\) be a small constant to be fixed subsequently, and for any $x\in B_{\rho-N}(0)$, we partition $\R^n$ into two regions:
\[
X_{1} := \left\{ y \in \mathbb{R}^{n} : |y - x| \geq (1 - \sigma) ||x| - \rho| \right\},
\]
and
\[
X_{2} := \left\{ y \in \mathbb{R}^{n} : |y - x| \leq (1 - \sigma) ||x| - \rho| \right\}.
\]

From \eqref{Green} \eqref{tildef}, we have
\[
\begin{aligned} 
\int_{X_{1}} G(x, y) J(y) dy & \leq C \|J\|_{L^{\infty}} \int_{(1 - \sigma)(\rho - |x|)}^{\infty} e^{-\lambda_{2} r} r^{n-1} dr \\ 
& \leq C \left( 1 + C_{1}^{2 \wedge p} \right) \int_{(1 - \sigma)(\rho - |x|)}^{\infty} e^{-\lambda_{2} r} r^{n-1} dr.
\end{aligned}
\]
Let \(\sigma\) be a fixed small number such that \(\lambda_{3} < (1 - \sigma) \lambda_{2}\), then it holds that 
\[
\int_{(1 - \sigma)(\rho - |x|)}^{\infty} e^{-\lambda_{2} r} r^{n-1} dr \leq C e^{-\lambda_{3}(\rho - |x|)}.
\]
By virtue of the last two equations, we obtain that 
\be\label{A1jifen}
\int_{X_{1}} G(x, y) J(y) dy \leq C \left( 1 + C_{1}^{2 \wedge p} \right) e^{-\lambda_{3}(\rho - |x|)}, \text{ for } |x|\leq (\rho-N). 
\ee

Since
\[
|y - x| + |\rho - |y|| \geq |\rho - |x||, \text{ for } y \in X_{2},
\]
and by means of \eqref {Green} and \eqref {tildef}, we show that
\[
\begin{aligned} 
\int_{X_{2}} G(x, y) J(y) dy & \leq C \left( 1 + C_{1}^{2 \wedge p} \right) \int_{X_{2}} e^{-\lambda_{2}|x - y|} e^{-\lambda_{2}|\rho - |y||}dy \\ 
& \leq C \left( 1 + C_{1}^{2 \wedge p} \right) \int_{X_{2}} e^{-\lambda_{2}|\rho - |x||}dy.
\end{aligned}
\]
Since \(|X_{2}| \leq C(1 - \sigma)^{n} |\rho - |x||^{n}\), we can verify
\be\label{A2jifen}
\int_{X_{2}} G(x, y) J(y) dy \leq C \left( 1 + C_{1}^{2 \wedge p} \right) e^{-\lambda_{3}(\rho - |x|)}.
\ee
As a result, from \eqref{tildeu2}, \eqref{A1jifen} and \eqref{A2jifen}, we conclude that 
\[
|\tilde{\omega}_2|\leq\tilde{u}_{2}(x) \leq C \left( 1 + C_{1}^{2 \wedge p} \right) e^{-\lambda_{3}(\rho - |x|)}, \text{ for } x\in B_{\rho-N}(0).
\]

In light of \eqref{tildew1guji}, we have
\be\label{tildewdezuizhongguji}
|\tilde{\omega}(r)| 
\leq \overline{C}_{0} \left( 1 + C_{1}^{2 \wedge p} \right) e^{-\lambda_{3}(\rho - r)} + \tilde{C} \varepsilon^3 e^{-\lambda_{1}(\rho-r)}, \text{ for } 0\leq r\leq(\rho-N).\ee

Fixing \(\gamma > 2 C_{0}\) large enough, with $C_0$ is given in \eqref{sepwdefanshu}. So that for \(\varepsilon\) sufficiently small, using \eqref{sepwdefanshu}, we have shown
\begin{equation*}
\|\tilde{\omega}\| =\|\mathcal{A}_{\varepsilon}(\omega)\|\leq \gamma \varepsilon^3 \|z\|,
\end{equation*}
and
\be\label{budengshi}
\overline{C}_{0} \left( 1 + C_{1}^{2 \wedge p} \right) e^{-(\lambda_{3} - \lambda_{1}) N} + \tilde{C} \varepsilon^3 \leq \gamma,
\ee
where \(\overline{C}_{0}\) is given in \eqref{tildewdezuizhongguji}. Combining \eqref{tildewdezuizhongguji} and \eqref{budengshi}, we see that
\begin{align*}
\bs
|\tilde{\omega}(r)|&\leq \overline{C}_{0} \left( 1 + C_{1}^{2 \wedge p} \right) e^{-(\lambda_{3} - \lambda_{1})(\rho - r)} e^{-\lambda_{1}(\rho - r)} +\tilde{C}\e^3e^{-\lambda_1(\rho-r)}\\
&=e^{-\lambda_1(\rho-r)}\left[\overline{C}_{0} \left( 1 + C_{1}^{2 \wedge p} \right) e^{-(\lambda_{3} - \lambda_{1})(\rho - r)}+\tilde{C}\e^3\right]\\&\leq e^{-\lambda_{1}(\rho - r)}\left[\overline{C}_{0} \left( 1 + C_{1}^{2 \wedge p} \right) e^{-(\lambda_{3} - \lambda_{1})N}+\tilde{C}\e^3\right]\\
&\leq\gamma e^{-\lambda_{1}(\rho - r)}, \text{ for } 0\leq r\leq(\rho-N).
\es
\end{align*}

Alternatively, owing to \(\rho \in \Omega_{\varepsilon}\), \eqref{jingxiangjiedeguji} and the fact that \(\|\tilde{\omega}\| \leq \gamma \varepsilon^3 \|z\|\) imply, for $1<(\rho-N)\leq r\leq \rho$,
\begin{align}
\label{>rho-N}\bs
|\tilde{\omega}(r)|&\leq Cr^{\frac{1-n}{2}}\|\tilde{\omega}\| \leq C(\rho-N)^{\frac{1-n}{2}} \gamma \varepsilon ^3\|z\|\\
&\leq C(\rho-N)^{\frac{1-n}{2}} \gamma \varepsilon ^3\e^{\frac{3-3n}{2}} \leq C\gamma \varepsilon ^3\leq\gamma e^{-\lambda_1(\rho-r)},
\es\end{align}
provided \(\varepsilon\) sufficiently small such that $C\e^3\leq e^{-\lambda_1N}$ in \eqref{>rho-N}.

Finally, we have that for any $\omega\in E_{\varepsilon}$, 
$$|\tilde{\omega}(r)| \leq \gamma e^{-\lambda_{1}(\rho - r)},\,\,0\leq r\leq \rho.$$

\smallskip
By step 1 and step 2, \(\mathcal{A}_{\varepsilon}\) is a contraction from \(E_{\varepsilon}\) to itself. Thus, there exists \(\omega \in E_{\varepsilon}\) such that \(\omega = \mathcal{A}_{\varepsilon}(\omega)\), solving i), ii) and iii). As for iv), it can be proved from \eqref{sepwdefanshu}.

The $C^{1}$–dependence of $\omega$ on $\rho$, as well as the corresponding estimate for 
$\bigl\|\tfrac{\partial \omega}{\partial \rho}\bigr\|$, can be obtained by arguments analogous to those in
\cite{cmp}; we therefore omit the details.
This completes the proof of Proposition~\ref{prop4.3}.

\end{proof}

\section{proof of theorem \ref{thm1.1}}
\subsection{Existence of solutions to \eqref{disanjiediyige}}\label{5.1}
To begin with, we establish the following Lemma.

\begin{lemma} \label{lem5.1}
There exists a constant \(C_{0} > 0\), for \(\varepsilon > 0\) small, the following holds
\[
\varepsilon^{3n-3} J_{\varepsilon}(z_{\rho, \varepsilon} + \omega_{\rho, \varepsilon}) = C_{0} \e^2M_\e(\varepsilon \rho) + o(1), \quad \rho \in \Omega_\e.
\]
\end{lemma}

\begin{proof}
For conciseness, we denote $z$ in place of $z_{\rho,\e}$ and $\omega$ in place of $\omega_{\rho,\e}$. Because
\[
J_{\varepsilon}(z + \omega) = J_{\varepsilon}(z) + J_{\varepsilon}'(z)[\omega] + \int_{0}^{1} J_{\varepsilon}''(z + s \omega)[\omega]^{2} ds,
\]
and combining \eqref{wfanshu}, (A1) and (A2), one can prove that
\be\label{O}
J_{\varepsilon}(z + \omega) = J_{\varepsilon}(z) + O\left( \varepsilon^{9 - 3n} \right).
\ee

The condition \(z\) concentrates near \(\rho\) yields 
\be\label{jifenfenjie}
\begin{aligned} 
J_{\varepsilon}(z) & = \int_{0}^{\infty} r^{n-1} \left( \frac{|z'|^{2} + (1+\e^2V(\varepsilon r)) z^{2}}{2} - \frac{z^{p+1}}{p+1} \right) dr \\ 
& = \rho^{n-1} \int_{\mathbb{R}} \left( \frac{|U'|^{2} +(1+\e^2 V(\varepsilon \rho)) U^{2}}{2} - \frac{U^{p+1}}{p+1} \right) dr (1 + o(1)).
\end{aligned}
\ee
Combining with the fact
\[
U(r) = U_{\rho,\e}(r)=(1+\e^2V(\e\rho))^{\frac{1}{p-1}} Q((1+\e^2V(\e\rho))^{\frac{1}{2}} r),
\]
we have
\be\label{Ujifen}
\int_{\mathbb{R}} \left( \frac{|U'|^{2} +(1+\e^2 V(\varepsilon \rho)) U^{2}}{2} - \frac{U^{p+1}}{p+1} \right) dr = C_{0} (1+\e^2V(\e\rho))^{\frac{p+3}{2(p-1)}},
\ee
where 
\[
C_{0} = \left( \frac{1}{2} - \frac{1}{p+1} \right) \int_{\mathbb{R}} Q^{p+1}(t)dt.
\]

Substituting \eqref{jifenfenjie} and \eqref{Ujifen} into the preceding equation \eqref{O}, we obtain
\begin{align*}
J_{\varepsilon}(z + \omega) &=\rho^{n-1}C_0(1+\e^2V(\e\rho))^{\frac{p+3}{2(p-1)}}(1+o(1))+O(\e^{9-3n})\\
&=\rho^{n-1}C_0(1+\e^2V(\e\rho))^{\frac{p+3}{2(p-1)}}+o(\e^{3-3n})+O(\e^{9-3n}).   
\end{align*}
In light of the definition of \(M_\e(r)\) as defined in \eqref{Mepr}, we get 
\begin{align*}
J_{\varepsilon}(z + \omega) &= \frac{C_{0}}{\varepsilon^{n-1}\e^{2n-4}} \e^{2n-4}(\varepsilon \rho)^{n-1} (1+\e^2V(\e\rho))^{\frac{p+3}{2(p-1)}}+o\left( \varepsilon^{3 - 3n}\right)+O(\e^{9-3n})\\ 
&= \frac{C_{0}}{\varepsilon^{3n-5}} M_\e(\varepsilon \rho) +o\left( \varepsilon^{3 - 3n}\right)+O(\e^{9-3n}),    
\end{align*}
which is equivalent to
\begin{equation}
\label{epjieshu}
\e^{3n-3}J_{\varepsilon}(z+\omega)=C_0\e^2M_\e(\e\rho)+o(1)+O(\e^6)=C_0\e^2M_\e(\e\rho)+o(1),
\end{equation}
and Lemma \ref {lem5.1} is  established.
\end{proof}

\smallskip
\begin{Rem}
Regarding \eqref{epjieshu}, we explain here why the term $C_{0}\varepsilon^{2} M_\varepsilon(\varepsilon \rho)$ is of order $\varepsilon^{0}$. By the definition of $M_\varepsilon(r)$, we have
\[
C_{0}\varepsilon^{2} M_\varepsilon(\varepsilon \rho)
= C_{0}\varepsilon^{3n-3}\rho^{\,n-1}
\bigl(1+\varepsilon^{2}V(\varepsilon \rho)\bigr)^{\frac{p+3}{2(p-1)}}.
\]
Since $\rho \sim \varepsilon^{-3}$, or $\rho^{\,n-1}\sim \varepsilon^{-3(n-1)}$, then
\(
\varepsilon^{3n-3}\rho^{\,n-1}\sim \varepsilon^{3n-3}\varepsilon^{-3(n-1)}
= \varepsilon^{0}.
\)
Therefore, $C_{0}\varepsilon^{2} M_\varepsilon(\varepsilon \rho)$ is indeed of order $\varepsilon^{0}$.

\end{Rem}
\medskip

\begin{proof}[\textbf{Proof of Theorem \ref{thm1.1} without constraint:}]

(i) We first consider the range of the exponent $p$ lying in 
\(
\left(1,\frac{n+2}{n-2}\right].
\)

Applying Lemma~\ref{lem5.1} and the condition $M_\varepsilon'(t_\varepsilon) = 0$, 
it follows that 
\[
\Psi_{\varepsilon}(\rho) := J_{\varepsilon}\big(z_{\rho, \varepsilon} + \omega_{\rho, \varepsilon}\big)
\]
admits a stationary point at some $\rho_{\varepsilon} \sim t_\varepsilon/\varepsilon$ with 
$\rho_{\varepsilon} \in \Omega_{\varepsilon}$. Invoking Proposition~\ref{prop2.2}, a stationary 
point of $\Psi_{\varepsilon}$ of this form yields a critical point 
\[
\tilde{u}_{\varepsilon}(r) 
= z_{\rho_{\varepsilon}, \varepsilon}(r)+ \omega_{\rho_{\varepsilon}, \varepsilon}(r)
\]
of the functional $J_\varepsilon$, which constitutes a radial solution to \eqref{disanjiediyige}. 
By a rescaling argument, we deduce that 
\[
u_{\varepsilon}(r) := \tilde{u}_{\varepsilon}(r/\varepsilon)
\]
is also a radial solution to \eqref{fangcheng1}. Since 
\[
\tilde{u}_{\varepsilon}(r) \sim U(r - \rho_{\varepsilon}) 
\sim U\big(r - t_\varepsilon/\varepsilon\big),
\]
we have 
\[
u_{\varepsilon}(r) \sim U\big((r - t_\varepsilon)/\varepsilon\big),
\]
which implies that $u_{\varepsilon}$ concentrates near the sphere $\lvert x\rvert = t_\varepsilon$.

\smallskip
(ii) In the supercritical case, where $p > \frac{n+2}{n-2}$, the proof proceeds by
introducing a truncation of the nonlinear term and then deriving a priori
$L^{\infty}$-estimates for the corresponding solutions. The argument of the
previous case can then be adapted with only minor modifications.

For some \(K>0\), we construct a smooth positive function
\(Y_{K}:\mathbb{R}\to\mathbb{R}\) with the following key properties:
\[
Y_{K}(t)=|t|^{p+1} \quad \text{for } |t|\le K,
\qquad
Y_{K}(t)=(K+1)^{p+1} \quad \text{for } |t|\ge K+1.
\]
We then define the truncated functional \(J_{\varepsilon,K}:H_{r}^{1}\to\mathbb{R}\)
by replacing \(|u|^{p+1}\) with \(Y_{K}(u)\) in the definition of \(J_{\varepsilon}\).

First, we observe that there exists a constant \(K_{0}>0\) such that
\(\|z_{\rho,\varepsilon}\|_{L^{\infty}}\le K_{0}\) for all \(\rho\in\Omega_{\varepsilon}\)
and all sufficiently small \(\varepsilon\). Consequently, if we fix
\(K\ge K_{0}\), the operator \(P J_{\varepsilon,K}''(z)\) is invertible and its
inverse \(L_{\varepsilon,K}\) has a norm that is uniformly bounded, independently
of \(K\).

Moreover, if \(K\) is chosen large enough, then, in combination with the
pointwise bounds on \(|\omega(r)|\) in \eqref{wzhudian}, we obtain that the
estimates for \(z+\omega\) are also uniform with respect to \(K\). As a result,
the truncated problem yields a genuine solution of \eqref{disanjiediyige}
for \(\varepsilon\) sufficiently small.

\end{proof}
\subsection{With constraint \eqref {disanjiedierge}.}
\begin{proof}[\textbf{Proof of Theorem \ref{thm1.1} completed:}]

Next, we solve the \eqref {disanjiediyige} with constraint condition \eqref{disanjiedierge}. From subsection \ref{5.1}, we can derive that there exists a small constant $\varepsilon_0>0$, for any $0<\e\leq\e_0$, $\tilde{u}_\e(r)=z_{\rho_\e,\e}(r)+\omega_{\rho_\e,\e}(r)$ is a radial concentrated solution of \eqref{disanjiediyige}.

It is obvious that \eqref{disanjiedierge} is equivalent to 
$$\int_0^{+\infty}r^{n-1}\tilde{u}^2_\e(r)dr=\frac{a^{\frac{2}{p-1}}\e^{\frac{4}{p-1}-n}}{\omega_{n-1}}.$$
Combining with $\tilde{u}_\e(r)=z_{\rho_\e,\e}(r)+\omega_{\rho_\e,\e}(r)$, we get
$$\int_0^{+\infty}r^{n-1}\tilde{u}^2_\e(r)dr=\int_0^{+\infty}r^
{n-1}(z^2_{\rho_\e,\e}(r)+2z_{\rho_\e,\e}(r)\omega_{\rho_\e,\e}(r)+\omega^2_{\rho_\e,\e}(r))dr,
$$
where
\begin{align}
\bs
\label{zw}
\int_0^{+\infty}r^
{n-1}z_{\rho_\e,\e}(r)\omega_{\rho_\e,\e}(r)dr&\leq C \|z_{\rho_\e,\e}(r)\|\|\omega_{\rho_\e,\e}(r)\|\\
&\leq C \|z_{\rho_\e,\e}(r)\|\gamma\e^3  \|z_{\rho_\e,\e}(r)\|\\
&=o(\|z_{\rho_\e,\e}(r)\|^2),\es
\end{align}
and
\be\label{omega^2}
\int_0^{+\infty}r^
{n-1}\omega^2_{\rho_\e,\e}(r)dr\leq C \|\omega_{\rho_\e,\e}(r))\|^2=o(\|z_{\rho_\e,\e}(r)\|^2).
\ee

In addition, since $z_{\rho_\e,\e}(r)=\zeta_{\varepsilon}(r)U_{\rho_\e,\e}(r-\rho_\e)$ and $$U_{\rho,\e}(r)=(1+\e^2V(\e\rho))^{\frac{1}{p-1}}Q((1+\e^2V(\e\rho))^{\frac{1}{2}}r),$$ it is easy to check that
\begin{align}
\label{z^2}\begin{split}
\int_0^{+\infty}r^
{n-1}z^2_{\rho_\e,\e}(r)dr&=\rho_\e^{n-1}\int_{\R}U_{\rho_\e,\e}^2(r)dr(1+o(1))\\
&=\rho_\e^{n-1}\int_{\R}(1+\e^2V(\e\rho_\e))^{\frac{2}{p-1}}Q^2((1+\e^2V(\e\rho_\e))^{\frac{1}{2}}r)dr(1+o(1))\\
&=\rho_\e^{n-1}(1+\e^2V(\e\rho_\e))^{\frac{2}{p-1}-\frac{1}{2}}\int_{\R}Q^2(t)dt(1+o(1))\\
&=\rho_\e^{n-1}\int_{\R}Q^2(t)dt(1+o(1))\\
&=C_\e^{n-1}\e^{3-3n}(1+o(1))\int_{\R}Q^2(t)dt,
\end{split}\end{align}
where $(1+\e^2V(\e\rho_\e))^{\frac{2}{p-1}-\frac{1}{2}}=1+o(1)$ and $\frac{1}{2}C_1\leq C_\e\leq2C_2$, according to the fact that $\frac{1}{2}C_1\e^{-3}\leq\rho_\e\leq2C_2\e^{-3}$.

As in \eqref{zw}, \eqref{omega^2} and \eqref{z^2}, it follows that  \eqref{disanjiedierge} is equivalent to 
\begin{equation*}
C_\e^{n-1}\e^{3-3n}\int_{\R}Q^2(t)dt+o(\e^{3-3n})=\frac{a^{\frac{2}{p-1}}\e^{\frac{4}{p-1}-n}}{\omega_{n-1}},
\end{equation*}
and this equivalence implies,
\begin{equation*}
BC_\e^{n-1}\e^{3-3n}+o(\e^{3-3n})-a^{\frac{2}{p-1}}\e^{\frac{4}{p-1}-n}=0,
\end{equation*}
with $B:=\omega_{n-1}\int_{\R}Q^2(t)dt>0$. According to Lemma \ref{Lemma2.10}, for $a>0$ large enough, actually with $a\geq \left(\frac{2^nC_2^{n-1}B}{\var_0^{\frac{4}{p-1}-3+2n}}\right)^{\frac{p-1}{2}}$, we define
\begin{equation*}
F(\e):= BC_\e^{n-1}\e^{3-3n}-a^{\frac{2}{p-1}}\e^{\frac{4}{p-1}-n}+o\left(\e^{3-3n}\right).    
\end{equation*}
By direct computations, it holds
\begin{equation*}
\begin{split}
F(\var)\leq& 2^{n-1}C_2^{n-1}B\var^{3-3n}-a^{\frac{2}{p-1}}\var^{\frac{4}{p-1}-n}+o\left(\var^{3-3n}\right),\\
F(\var)\geq&\left(\frac{1}{2}\right)^{n-1}C_1^{n-1}B\var^{3-3n}-a^{\frac{2}{p-1}}\var^{\frac{4}{p-1}-n}+o\left(\var^{3-3n}\right).    
\end{split}  
\end{equation*}

Now let $\varepsilon_1=\left(\frac{2^nC_2^{n-1}B}{a^{\frac{2}{p-1}}}\right)^{\frac{1}{\frac{4}{p-1}-3+2n}}$,  it is easy to check that
\begin{equation*}
\begin{split}
F(\varepsilon_1)\leq
&\var_1^{3-3n}\left(2^{n-1}C_2^{n-1}B-a^{\frac{2}{p-1}}\var_1^{\frac{4}{p-1}-3+2n}\right)+o\left(\var_1^{3-3n}\right)\\
=&\var_1^{3-3n}\left(-2^{n-1}C_2^{n-1}B+o(1)\right)<0.
\end{split}
\end{equation*}
Similarly, taking $\varepsilon_2=\left(\frac{\left(\frac{1}{2}\right)^nC_1^{n-1}B}{a^{\frac{2}{p-1}}}\right)^{\frac{1}{\frac{4}{p-1}-3+2n}}$,
we observe that
\begin{equation*}
\begin{split}
F(\varepsilon_2)\geq&\var_2^{3-3n}\left(\left(\frac{1}{2}\right)^{n-1}C_1^{n-1}B-a^{\frac{2}{p-1}}\var_2^{\frac{4}{p-1}-3+2n}\right)+o\left(\var_2
^{3-3n}\right)\\=&\var_2^{3-3n}\left(\left(\frac{1}{2}\right)^{n}C_1^{n-1}B+o(1)\right)>0.      \end{split}
\end{equation*}
Then there exists $\varepsilon\in (\varepsilon_2,\varepsilon_1)$,  solving $F(\varepsilon)=0$. For any $a\geq \left(\frac{2^nC_2^{n-1}B}{\var_0^{\frac{4}{p-1}-3+2n}}\right)^{\frac{p-1}{2}}$, equation \eqref{disanjiedierge} has a solution $\varepsilon_a\in\left(\left(\frac{\left(\frac{1}{2}\right)^nC_1^{n-1}B}{a^{\frac{2}{p-1}}}\right)^{\frac{1}{\frac{4}{p-1}-3+2n}},\left(\frac{2^nC_2^{n-1}B}{a^{\frac{2}{p-1}}}\right)^{\frac{1}{\frac{4}{p-1}-3+2n}}\right)$. That is to say, as $a\to+\infty$, there exists a sequence of radial solutions $\tilde{u}_{\e_a}(r)=z_{\rho_{\e_a},\e_a}(r)+\omega_{\rho_{\e_a},\e_a}(r)$
to equations \eqref{disanjiediyige} and \eqref{disanjiedierge}. Through scaling transformations, it can be derived that $u_{\e_a}(r)=\tilde{u}_{\e_a}\left(\frac{r}{\e_a}\right)$
is a radial solution to equations \eqref{fangcheng1}-\eqref{fangcheng11} and \(u_{\varepsilon_a}(r)\) concentrates near the sphere \(|x| = t_{\e_a}\). This concludes the proof of Theorem \ref{thm1.1}.
\end{proof}
\medskip
\noindent\textbf{Acknowledgments}:  
This research was supported by the National Natural Science Foundation of China (Grant Nos. 12271539). The authors are very grateful to Professor Jun Yang for many helpful discussions and suggestions.

\end{document}